\documentclass{amsart}
\usepackage[all]{xy}
\usepackage[OT2,T1]{fontenc}
\DeclareSymbolFont{cyrletters}{OT2}{wncyr}{m}{n}
\DeclareMathSymbol{\Sha}{\mathalpha}{cyrletters}{"58}
\newtheorem{theorem}{Theorem}[section]
\newtheorem{lemma}[theorem]{Lemma}
\newtheorem{corollary}[theorem]{Corollary}
\newtheorem{proposition}[theorem]{Proposition}
\newtheorem{conjecture}[theorem]{Conjecture}

\theoremstyle{definition}

\newtheorem{example}[theorem]{Example}

\usepackage{tikz}

\theoremstyle{remark}
\newtheorem{remark}[theorem]{Remark}

\newcommand{\C}{\mathbb C}
\newcommand{\R}{\mathbb R}

\newcommand{\Q}{\mathbb Q}
\newcommand{\G}{\mathbb G}
\newcommand{\Z}{\mathbb Z}

\newcommand{\F}{\mathbb F}
\newcommand{\A}{\mathbb A}

\begin{document}

\title{{Wild Kernels and divisibility in K-groups of global fields}}

\author[Grzegorz Banaszak]{Grzegorz Banaszak}
\address{Department of Mathematics and Computer Science, Adam Mickiewicz University, 
Pozna\'{n} 61614, Poland}
\email{banaszak@amu.edu.pl}



\subjclass[2000]{19D10, 11G30}
\date{}
\keywords{K-theory, Galois cohomology, global field, divisible elements, wild kernel}

\thanks{Partially supported by
the NCN (National Center for Science of Poland) 
grant NN201607440}

\begin{abstract}
{In this paper we study the divisibility and the wild kernels 
in algebraic K-theory of global fields $F.$ We extend the notion of the wild kernel
to all K-groups of global fields and prove that Quillen-Lichtenbaum conjecture for $F$
is equivalent to the equality of wild kernels with corresponding groups of 
divisible elements in K-groups of $F.$ We show that there exist generalized Moore exact 
sequences for even K-groups of global fields. Without appealing to the Quillen-Lichtenbaum 
conjecture we show that the group of divisible elements is isomorphic to the corresponding 
group of \' etale divisible elements and we apply this result for the proof of the $lim^1$ 
analogue of Quillen-Lichtenbaum conjecture. We also apply this isomorphism to investigate:
the imbedding obstructions in homology of $GL,$ 
the splitting obstructions for the Quillen localization sequence, the order of the group of
divisible elements via special values of $\zeta_{F}(s).$ Using the motivic cohomology results due 
to Bloch, Friedlander, Levine, Lichtenbaum, Morel, Rost, Suslin, Voevodsky and Weibel, 
which established the Quillen-Lichtenbaum conjecture, we conclude that wild kernels are equal to corresponding
groups of divisible elements.}
\end{abstract}

\maketitle


\section{Introduction}
Let $l$ be a prime number and let $F$ be a global field of characteristic 
$\text{char}\, F  \not= l.$ If $l = 2$ we assume that $\mu_{4} \subset F.$
The main goal of this paper is to establish general results concerning divisibility and 
wild kernels in algebraic K-theory of global fields. 
\medskip

It has already been shown by 
Bass \cite{B}, Tate \cite{Ta2} and Moore (see \cite[p. 157]{Mi}) that for a
number field $F$ the group $K_2 (F)$ and 
in particular the group of divisible elements and the wild kernel in $K_2 (F)$ are closely related 
to arithmetic of $F$ and the Dedekind zeta $\zeta_{F} (s)$ at $s = -1.$ The divisible 
elements for the Galois cohomology of number fields and local fields in the mix characteristic 
case were introduced in \cite{Sch2}. The divisible elements and wild kernels for the odd torsion part 
for the even higher K-groups of number fields were introduced in \cite{Ba1} and \cite{Ba2}. The results of 
Bass \cite{B}, and Moore (see \cite[p. 157]{Mi}) concerning the divisible elements and the wild kernel 
for $K_2$ where extended in \cite{Ba2} to higher even K-groups of number fields $F$ and values of $\zeta_{F} (s)$ 
at negative odd integers. The {\' e}tale wild kernel as a Shafarevich group in Galois cohomology of number fields
was introduced in \cite{Ng}. The work in \cite{Ba1}, \cite{Ba2}, 
\cite{Ng} and \cite{Sch2} was carried out under assumption $l > 2.$ The wild kernels for the $2$-primary part for 
the higher, even K-groups of number fields were introduced in \cite{Os} and in \cite{We3} were studied for all $l \geq 2.$ 
\bigskip

In this paper we investigate the wild kernels and the divisible elements for even and odd K-groups of global 
fields and for all $l \geq 2$ 
(see Theorems \ref{Theorem 1.1} - \ref{Theorem 1.11} and Corollary \ref{Theorem 1.12} in this introduction for the
statement of main results). These results are new in the ${\rm{char}} \, F > 0$ case and 
some of them are new in the ${\rm{char}} \, F = 0$ case for $l \geq 2.$ Some of these results have already
been known in the ${\rm{char}} \, F = 0$ case, often for $l > 2,$ so we make corresponding references 
in this introduction. 
The Theorems \ref{Theorem 1.1} - \ref{Theorem 1.8},  Theorems 
\ref{Theorem 1.10} - \ref{Theorem 1.11}
and Corollary \ref{Theorem 1.12} are proven without appealing to the Quillen-Lichtenbaum conjecture.
The Theorem \ref{Theorem 1.9} is a consequence 
of the Quillen-Lichtenbaum conjecture that in turn resulted from the Voevodsky results \cite{V1}, 
\cite{V2} and the motivic cohomology results due to Bloch, Friedlander, 
Levine, Lichtenbaum, Morel, Rost, Suslin, Voevodsky, Weibel, and others
(see e.g. \cite{BL} cf. \cite[Appendix B]{RW}, \cite{L}, \cite{MV}, \cite{R}, \cite{VSF}, \cite{We4})
\medskip

It was shown in \cite{Ba2} and \cite{BGKZ} in the case of number fields that the Quillen-Lichtenbaum 
conjecture for odd torsion part of the even K-groups is equivalent to the isomorphism between corresponding 
wild kernels and divisible elements. 
C. Weibel \cite{We3} and K. Hutchinson \cite{Hu} worked on wild kernels and divisible 
elements in the K-theory of number fields for $l \geq 2.$ 
C. Weibel \cite{We3} computed the index of the the group of divisible elements in the 
corresponding wild kernel for even K-theory of number fields for the 2-primary part.
The key ingredients in his proof were the results on motivic cohomology that led to the
computations of the $2$-primary part of Quillen K-theory (cf. \cite{RW}).
In this paper we show (see Theorem \ref{Theorem 1.9}) that the wild kernel is isomorphic to the 
divisible elements in K-groups of global fields for all indexes $n > 1$ and $l \geq 2$ (assuming
$\mu_4 \in F$ in the case of the $2$-primary part). To prove this we show that, under 
our assumptions, the Quillen-Lichtenbaum conjecture is equivalent to the isomorphism between wild 
kernel and the divisible elements. In the number field case Theorem \ref{Theorem 1.9} 
follows from \cite{Ba2} (for $l$ odd) and \cite{We3}, \cite{Hu} (for $l = 2$). 
\medskip

Recall that the divisible elements in K-groups of number fields are in the center 
of classical conjectures in algebraic number theory and algebraic K-theory. 
Indeed, the conjectures of Kummer-Vandiver and Iwasawa can be reformulated in 
terms of divisible elements in even K-groups of $\Q$ \cite{BG1},  \cite{BG2}. 
We have already pointed out in \cite[p. 292]{Ba2}, that the group of divisible 
elements in an even K-group of a number field $F$ is an analogue of the class group 
of $\mathcal{O}_F$. Moreover, as shown in section 6 of this paper, there is a positive 
integer $N_{0}$ 
such that for every positive integer $N,$ such that $N_{0} \, | \, N,$ 
there is the following exact sequence:
\begin{equation}\small{
0 \rightarrow K_{2n} ({\mathcal O}_F) \rightarrow 
K_{2n} (F) [N] 
\rightarrow \bigoplus_{v} \, K_{2n-1} (k_v)[N] \rightarrow 
D (n) \rightarrow 0.}
\nonumber\end{equation} 
Recall that the class group $Cl ({\mathcal O}_F)$ appears in the classical 
exact sequence:
\begin{equation}\small{
0 \rightarrow K_{1} ({\mathcal O}_F) \rightarrow 
K_{1} (F) 
\rightarrow \bigoplus_{v} \, K_{0} (k_v) \rightarrow 
 Cl ({\mathcal O}_F) \rightarrow 0.}
\nonumber\end{equation} 
In addition, as already mentioned above, the conjecture of Quillen-Lichtenbaum can be reformulated 
in terms of the wild kernels and divisible elements (see also Theorem \ref{QLequivalent} 
for more detailed presentation). At the last but not the least we would like to point out 
the close relation of divisible elements and the Coates-Sinnott conjecture \cite{Ba1}, \cite{BP}.
\bigskip

The organization of the paper is as follows.
In chapter 2 we introduce basic notation and recall some classical
facts about the cohomological dimension. 
In chapter 3 we extend results of P. Schneider \cite{Sch2} concerning the divisible
elements in Galois cohomology (in \cite{Sch2} the 
assumption was $\text{char} \, F = 0$ and $l > 2$). Namely we obtain 
Theorem \ref{functionfieldSchneider3} and Corollary \ref{DIVandSZA} which lead us to the following  
analog of the classical Moore exact sequence (see \cite[p. 157]{Mi}) for higher {\' e}tale K-theory: 

\begin{theorem} \label{Theorem 1.1} 
Let $n \geq 1.$ 
For every finite $S \supset S_{\infty, l}$ there are exact sequences: 
\begin{equation}
\small{0 \rightarrow D^{et}(n) \rightarrow
K_{2n}^{et}({\mathcal O}_{F, S})  
\rightarrow \bigoplus_{v \in S} \, W^{n} (F_v) 
\rightarrow W^{n} (F) \rightarrow
0.}
\nonumber\end{equation}
\begin{equation}
\small{0 \rightarrow D^{et}(n) \rightarrow
K_{2n}^{et}(F)_l 
\rightarrow \bigoplus_{v} \, W^{n} (F_v) 
 \rightarrow W^{n} (F) \rightarrow 0} 
\nonumber\end{equation}
where $D^{et}(n) := div \, K_{2n}^{et}(F)_l.$  In particular: 
\begin{equation}
\frac{|K_{2n}^{et}({\mathcal O}_{F, S})|}{| D^{et}(n)|} = 
\frac{{\bigl |}\prod_{v \in S} w_{n} (F_v) {\bigl |}_{l}^{-1} }{|w_{n} (F)|_{l}^{-1}}.
\nonumber\end{equation}
\end{theorem}
\medskip

In chapter 4 we investigate the divisibility in K-theory and {\' e}tale K-theory
of $F.$ Let $$D(n) := div\,  K_{2n} (F).$$ 
For every $k > 0$ define: 
$$D (n, l^k) := \text{ker} \,  (K_{2n} ({\mathcal O}_F, \,  \Z /l^k) 
\rightarrow K_{2n} (F, \,  \Z/l^k)),$$ 
$$D^{et} (n, l^k) := 
\text{ker} \,  (K_{2n}^{et} ({\mathcal O}_{F} [1/l], \,  \Z /l^k) 
\rightarrow K_{2n}^{et} (F, \,  \Z/l^k)).$$
\medskip

\noindent
The following result shows that the Dwyer-Friedlander homomorphism is an isomorphism
when restricted to the groups $D (n, l^k)$ and $div \, K_{2n}(F)_l:$
\begin{theorem}\label{Theorem 1.2}
If $l > 2$ then $\forall \, k \geq 1$ there is 
the following canonical isomorphism:
\begin{equation}
D (n, l^k) \cong
D^{et} (n, l^k)
\nonumber\end{equation}
If $l = 2$ then $\forall \, k \geq 2$ 
there is the following canonical isomorphism:
\begin{equation}
D (n, 2^k) \cong
D^{et} (n, 2^k)
\nonumber\end{equation} 
If $l \geq  2$ then there is the following isomorphism $D(n)_l \cong  D^{et} (n)$ 
or more explicitly
\begin{equation}
div \, K_{2n}(F)_l \,\, \cong \,\, div \, K_{2n}^{et} (F)_l
\nonumber\end{equation}  
\end{theorem}

\noindent
The last isomorphism of Theorem \ref{Theorem 1.2} extends \cite[Theorem 3]{Ba2}
which was my joint result with M. Kolster. The Theorem 3 of \cite{Ba2} concerned
the number field case and $l > 2.$ By Theorem \ref{Theorem 1.2} and Corollary \ref{DIVandSZA} 
(see section 3) the divisible elements are expressed in terms of 
Tate-Shafarevich groups for all $n > 0:$
\begin{equation}
D(n)_l \, \cong \, D^{et} (n) \, \cong \, D_{n+1} (F) \, = \,  
\Sha^{2}_{S} (F, \Z_l (n+1)) = \Sha^{2} (F, \Z_l (n+1)).
\nonumber\end{equation}
We also get the following $lim^{1}$ analogue of the Quillen-Lichtenbaum conjecture.

\begin{theorem}\label{Theorem 1.3}
For every $n \geq 1$ there is the following isomorphism:
\begin{equation}\small{
{\varprojlim_{k}}^{1} \, K_{n} (F,\, \Z / l^k) \,\, \stackrel{\cong}{\longrightarrow} \,\, 
{\varprojlim_{k}}^{1} \, K_{n}^{et} (F,\, \Z / l^k).} \label{lim 1 QL2}
\end{equation}
Moreover there is the following equality:
\begin{equation}\small{
{\varprojlim_{k}}^{1} \, K_{2n} (F,\, \Z / l^k) = 0} \label{lim 1 QL3}
\end{equation}
and the exact sequence:
\begin{equation}\small{
0 \rightarrow D (n)_l \rightarrow {\varprojlim_{k}}^{1} \, K_{2n+1} (F, \, \Z / l^k)
\rightarrow {\varprojlim_{k}}^{1} \bigoplus_{v} \, K_{2n} (k_v, \Z / l^k)
\rightarrow 0} \label{lim 1 QL4}
\end{equation}
\end{theorem}

\noindent
Theorem \ref{Theorem 1.3} for the number field case and $l > 2$ was proved in 
\cite{BZ1}. 
\medskip

In the end of chapter 4 we show that the natural maps:
\begin{equation}
H_{2n} (GL(\mathcal{O}_{F}), \, \Z / l^k)  \,\, \rightarrow \,\,
H_{2n} (GL(F), \, \Z / l^k)
\nonumber
\end{equation}  
are not injective in general. We have the following theorem:
\begin{theorem}\label{Theorem 1.4}
For every $n \geq 1,$ $k \geq 1$ and $l > n + 1 $ the kernel of the natural map 
\begin{equation}
H_{2n} (GL(\mathcal{O}_{F}), \, \Z / l^k)  \,\, \rightarrow \,\,
H_{2n} (GL(F), \, \Z / l^k)
\label{ker on homology OF to F 1. Introduction}
\end{equation}  
contains a subgroup isomorphic to $D (n, l^k).$
\end{theorem}

In particular we show that the maps:
\begin{equation}
H_{22} (GL(\Z), \, \Z / 691) \,\, \rightarrow \,\,
H_{22} (GL(\Q), \, \Z / 691)
\nonumber
\end{equation}
\begin{equation}
H_{30} (GL(\Z), \, \Z / 3617) \,\, \rightarrow \,\,
H_{30} (GL(\Q), \, \Z / 3617)
\nonumber
\end{equation}
are not injective. Moreover let $E/\F_p$ be the 
elliptic curve $y^2 = x^3 + 1.$ For $ p \geq 5$ and 
$p \equiv 2 \mod 3$ this curve is supersingular. 
If $F = \F_p (E)$ is the function field of $E$ then 
we show that the following maps are not injective:
\begin{equation}
H_{6} (GL(\mathcal{O}_{\F_{29} (E)}), \, \Z / 5) \,\, \rightarrow \,\,
H_{6} (GL(\F_{29} (E)), \, \Z / 5),
\nonumber 
\end{equation}  
\begin{equation}
H_{10} (GL(\mathcal{O}_{\F_{41} (E)}), \, \Z / 7) \,\, \rightarrow \,\,
H_{10} (GL(\F_{41} (E)), \, \Z / 7).
\nonumber 
\end{equation} 
For $ p \geq 3$ and 
$p \equiv 3 \mod 4$ the elliptic curve $y^2 = x^3 + x$ is supersingular.
In particular we show that the following map is not injective:
\begin{equation}
H_{6} (GL(\mathcal{O}_{\F_{19} (E)}), \, \Z / 5) \,\, \rightarrow \,\,
H_{6} (GL(\F_{19} (E)), \, \Z / 5).
\nonumber 
\end{equation}  
\bigskip

In chapter 5 we define wild kernel $K_{n}^{w} ({\mathcal O}_{F})_l$ for all $n > 0:$
\begin{equation}
K_{n}^{w} ({\mathcal O}_{F})_l \, := \, \text{ker} \, ( \, K_{n} (F)_l 
\rightarrow \bigoplus_{v} \, K_{n}^{et} (F_v)_l \, )
\nonumber\end{equation} 
and observe that: 
\begin{equation}
div \, K_{n} (F)_l \,  \subset  \, K_{n}^{w} ({\mathcal O}_{F})_l \, 
\subset \, K_{n} ({\mathcal O}_{F})_l.
\nonumber \end{equation}  
Further, we obtain the analogue of the classical Moore exact sequence for higher K-groups: 
\begin{theorem} \label{Theorem 1.5}
For every $n \geq 1$ and every finite set 
$S \supset S_{\infty, l}$ there are the following exact sequences:
\begin{equation}
\small{0 \rightarrow K_{2n}^{w} ({\mathcal O}_{F})_l \rightarrow
K_{2n} ({\mathcal O}_{F,S})_l
\rightarrow \bigoplus_{v \in S} W^n (F_v)
\rightarrow W^{n} (F) \rightarrow
0.}
\label{Moore exact sequence11}\end{equation}
\begin{equation}
\small{0 \rightarrow K_{2n}^{w} ({\mathcal O}_{F})_l \rightarrow
K_{2n} (F)_l
\rightarrow \bigoplus_{v} W^n (F_v)
\rightarrow W^{n} (F) \rightarrow
0.}
\label{Moore exact sequence22}\end{equation}
In particular:
\begin{equation}
\small{\frac{|K_{2n} (\mathcal{O}_{F, S})_l|}{|K_{2n}^{w} (\mathcal{O}_{F})_l|} 
= {\bigl|}\frac{\prod_{v \in S} w_{n} (F_v)}{w_{n} (F)}{\bigr|}_{l}^{-1}.}
\label{Kth OS comp to Wild}\end{equation}
\end{theorem}

\noindent
The exact sequences (\ref{Moore exact sequence11}) and (\ref{Moore exact sequence22})
where established for number fields in \cite{Ba2} ($l >2$) and \cite{We3} ($l=2$).
\noindent
In chapter 5 we define another wild kernel $WK_{n} (F)$ for all $n \geq 0:$
\begin{equation}
WK_{n} (F) \, := \, \text{ker} \, ( \, K_{n} (F) \rightarrow \bigoplus_{v} \, K_{n} (F_v) \, )
\nonumber\end{equation} 
We observe that for all $n \geq 0:$
\begin{equation}
WK_{n} (F) \subset K_{n} (\mathcal{O}_{F})_{tor}
\nonumber\end{equation}
\begin{equation}
WK_{n} (F)_l \subset K_{2n}^{w} ({\mathcal O}_{F})_l
\nonumber\end{equation}
and if $K_{n} (F_v)_l 
\stackrel{\cong}{\longrightarrow} K_{n}^{et} (F_v)_l $ 
for every $v \in S_l,$ then:
\begin{equation}
div \, K_{n} (F)_l \,  \subset  \, WK_{n} (F)_l \, 
\subset \, K_{n} ({\mathcal O}_{F})_l.
\nonumber \end{equation}

\noindent
The Dwyer-Friedlander homomorphisms \cite{DF} which are surjective: 
\begin{equation}
K_n (\mathcal{O}_{F, S})_l \rightarrow K_{n}^{et} (\mathcal{O}_{F, S})_l
\nonumber \end{equation}
\begin{equation}
K_n (F)_l \rightarrow K_{n}^{et} (F)_l,
\nonumber\end{equation}
(see also \cite{Ba2}) are also split as follows by results of \cite{Ba2}, \cite{Ca}, \cite{K}
which can be extended also to the function field case.
We use this to establish the following properties of wild kernels and 
divisible elements
(see Theorems  \ref{H2 of F}, \ref{wild maps on div via a split map 1} 
and \ref{wild maps on div via a split map 2}). 

\begin{theorem} \label{Theorem 1.6}
For all $n \geq 1$ the Dwyer-Friedlander homomorphisms induce the following 
canonical map: 
\begin{equation}
K_{n}^{w} ({\mathcal O}_{F})_l \rightarrow
div \, K_{n}(F)_l
\nonumber\end{equation}
which is split surjective. The Dwyer-Friedlander homomorphisms
induce the following canonical map:
\begin{equation}
WK_{n} (F)_l \rightarrow div \, K_{n}(F)_l
\nonumber\end{equation}
which is split surjective if $K_{n} (F_v)_l 
\stackrel{\cong}{\longrightarrow} K_{n}^{et} (F_v)_l $ 
for every $v \in S_l.$
\end{theorem}

\noindent
The first splitting map of the Theorem \ref{Theorem 1.6} in case of number fields 
and $l > 2$ was done in \cite{Ng} by use of an argument from \cite{Ba2}.
Note that Theorem \ref{Theorem 1.6} is obvious for $n$ odd since in this case 
$div \, K_{n}(F) = 0.$ 
\medskip

\noindent
Theorem \ref{Theorem 1.7} below shows that the Quillen-Lichtenbaum conjecture holds modulo the 
wild kernel:
\begin{theorem} \label{Theorem 1.7}
The Dwyer-Friedlander homomorphisms induce the following canonical isomorphisms for all 
$n \geq 1:$
\begin{equation}
K_{n}(\mathcal{O}_{F,S})_l / K_{n}^{w}(\mathcal{O}_{F})_l 
\,\, \stackrel{\cong}{\longrightarrow} \,\,
K_{n}^{et}(\mathcal{O}_{F,S})_l / div K_{n}^{et} (F)_l  
\nonumber\end{equation}
\begin{equation}
K_{n}(F)_l / K_{n}^{w}(\mathcal{O}_{F})_l 
\,\, \stackrel{\cong}{\longrightarrow} \,\,
K_{n}^{et}(F)_l / div K_{n}^{et} (F)_l  
\nonumber\end{equation}
\end{theorem} 
\medskip

\noindent
Theorem \ref{Theorem 1.8} below shows that the difference between the wild kernels 
and the divisible elements is the obstruction to the Quillen-Lichtenbaum conjecture. 
\begin{theorem}\label{Theorem 1.8} Let $n > 1.$ The following two conditions are equivalent:
\begin{equation} 
K_{n} (\mathcal{O}_{F}) \otimes \Z_l \stackrel{\cong}{\longrightarrow} 
K_{n}^{et} (\mathcal{O}_{F}[1/l]),
\nonumber\end{equation}
\begin{equation}
K_{n}^{w}(\mathcal{O}_{F})_l = div K_n (F)_l.  
\nonumber\end{equation} 
Moreover assume that $K_{n} (F_v)_l 
\stackrel{\cong}{\longrightarrow} K_{n}^{et} (F_v)_l$ for every $v \in S_l.$
Then the two conditions above are equivalent to: 
\begin{equation} 
WK_n(F)_l = div K_n (F)_l.
\nonumber\end{equation}  
\end{theorem}
\medskip

\noindent 
At the end of chapter 5, by use of the Rost-Voevodsky theorem, we prove: 

\begin{theorem}\label{Theorem 1.9} For every $n > 1$ we have the following equality: 
\begin{equation} 
K_{n}^{w}(\mathcal{O}_{F})_l = div K_n (F)_l. 
\nonumber
\end{equation} 
Assume that $K_{n} (F_v)_l 
\stackrel{\cong}{\longrightarrow} K_{n}^{et} (F_v)_l$ for every $v \in S_l$ and every
$n > 1.$ Then for every $n \geq 0:$  
\begin{equation} 
WK_n(F) = div\, K_n (F).
\nonumber\end{equation}  
\end{theorem}

\noindent
In the number field case the equality $K_{n}^{w}(\mathcal{O}_{F})_l = div K_n (F)_l,$ 
in Theorem \ref{Theorem 1.9}, follows from \cite{Ba2} (for $l$ odd) and \cite{We3}, \cite{Hu} (for $l = 2$). 
Observe that for $0 \leq n \leq 1$ we have $WK_n(F) = div\, K_n (F) = 0$ for obvious reasons. 
\bigskip

In chapter 6 we investigate the obstructions for the splitting of the Quillen localization
sequence and complete a statement of \cite[Cor. 1 and Prop. 1 p. 293]{Ba2}. 
Recall, that Tate (see \cite[Theorem 11.6]{Mi}) proved that there is the following isomorphism
$$K_{2} (\Q) \, \cong \, K_{2}(\Z) \, \oplus \,  \bigoplus_{p} \, K_{1} (\F_p).$$
The results concerning the splitting of the Quillen exact sequence for higher K-groups
of number fields were obtained in \cite{Ba1} and \cite{Ba2}. A very special case of results of \cite{Ba1} 
is the following isomorphism:
$$K_{2n} (\Q)_l \, \cong \, K_{2n}(\Z)_l \, \oplus \,  \bigoplus_{p} \, K_{2n-1} (\F_p)_l,$$
for $n = 3, 5, 7, 9$ and $l > 2.$ Moreover for $n$ odd and $l > 2$ 
the following conditions are equivalent
\cite[Cor. 2 p. 294]{Ba2} (see also Corollary \ref{splitting4} in section 6):

$$K_{2n} (\Q)_l \, \cong \, K_{2n}(\Z)_l \, \oplus \,  \bigoplus_{p} \, K_{2n-1} (\F_p)_l.$$
$$|{w_{n+1} (\Q) \zeta_{\Q} (-n)} |_{l}^{-1} = 1.$$   
\medskip

In this paper, under the assumption that $l \geq 2$ and $F$ is a global field with $\text{char}\, F  \not= l$ 
(if $l = 2$ we assume $\mu_{4} \subset F$), we get the following result concerning the splitting of the 
Quillen localization sequence that extends the splitting
results of \cite{Ba1}, \cite{Ba2} and \cite{Ca} in the number field case.
\newpage

\begin{theorem}\label{Theorem 1.10}
Let $n \geq 1.$ The following conditions are equivalent:
\begin{enumerate} 
\item{} $D (n, l^k) = 0$ \,\, for every \,\, $0 < k \leq k(l),$  
\item{} $D^{et} (n, l^k) = 0$ \,\, for every \,\, $0 < k \leq k(l),$  
\item{} $K_{2n} (F)_l \, \cong \, K_{2n} ({\mathcal O}_F)_l \, \oplus \,  
\bigoplus_{v} \, K_{2n-1} (k_v)_l,$
\item{} $K_{2n}^{et} (F)_l \, \cong \, K_{2n}^{et} ({\mathcal O}_{F}[1/l])_l \, \oplus \,  
\bigoplus_{v} \, K_{2n-1}^{et} (k_v)_l,$
\end{enumerate} 
where $k(l)$ is defined by (\ref{1.3}) and (\ref{1.4}) in section 4.
\end{theorem}

\noindent
The group of divisible elements is the obstruction to splitting of the following natural
boundary map in the Quillen localization sequence:
\begin{theorem}\label{Theorem 1.11}
Let $n > 0$ and let $k \geq k(l).$
The following conditions are equivalent:
\begin{enumerate} 
\item{}
$\partial_{1} \, : \, K_{2n} (F)_l
\rightarrow {\bigoplus_{{l^{k} \, | \, q_{v}^n - 1}}} \, K_{2n-1} (k_v)_l \quad\, \text{is split surjective},$
\item{} $D (n)_l = 0.$ 
\end{enumerate} 
\end{theorem}

\noindent
This implies the following corollary:
\begin{corollary}\label{Theorem 1.12}
Let $F$ be a totally real number field, $n$ odd and $l > 2$ or 
let $F$ be a global field of $\text{char} \, F > 0,$  $n \geq 1$
and $l \not= \, \text{char} \, F.$ Then for every $k \geq k(l)$ the 
following conditions are equivalent:
\begin{enumerate} 
\item{} The following surjective map splits
$$\partial_{1} \, : \, K_{2n} (F)_l
\rightarrow {\bigoplus_{{l^{k} \, | \, q_{v}^n - 1}}} \, K_{2n-1} (k_v)_l $$
\item{} 
$$\huge{|}\frac{w_{n} (F) \, w_{n+1} (F) \zeta_{F} (-n)}{\prod_{v \in S_{\infty, l}} \,  w_n (F_v)} \big{|}_{l}^{-1} = 1.$$  
\end{enumerate}
\end{corollary}
\bigskip

\noindent
Observe that $\big{|}  w_{n} (\R) \big{|}_{l}^{-1} = \big{|}  w_{n} (F) \big{|}_{l}^{-1} = 1$  for $F$ totally real, $n$ odd and $l$ odd.


\section{Basic notation and set up}

\subsection{Notation}
\begin{enumerate}
\item{} $l$ is a prime number.
\item{} $F \,\, := $  a global field. 
\item{} $p := \text{char} \, F,$ if $\text{char} \, F > 0.$   
\item{} $\mathcal{O}_F :=
\left\{
\begin{array}{lll}
\text{the integral closure of} \,\,\,\,  \Z \,\, \text{in} \,\, F &\rm{if} & \text{char} \, F = 0\\
\text{the integral closure of} \,\,\,\, \F_p [t] \,\, \text{in} \,\, F &\rm{if}& \text{char} \, F > 0\\
\end{array}\right.$
\item{} $v$ a place of $F.$
\item{} $S_{\infty} := \left\{
\begin{array}{lll}
\{v : \, v | \infty\} &\rm{if} & \text{char} \, F = 0\\
\{v : \, v | v_{t^{-1}}\} &\rm{if}& \text{char} \, F > 0\\
\end{array}\right.$
\item{}  $S_{l} := \left\{
\begin{array}{lll}
 \{v\, : \, v | l\} &\rm{if} & \text{char} \, F = 0\\
\emptyset &\rm{if}& \text{char} \, F > 0\\
\end{array}\right.$
\item{} $S_{\infty, l} := S_{\infty} \cup S_l.$
\item{} $S$ a finite set of places of $F$ containing $S_{\infty, l}.$
\item{} $\mathcal{O}_{F, S}$ the ring of $S$-integers in $F.$ Note that 
$\mathcal{O}_{F, S_{\infty}} = \mathcal{O}_{F}.$ 
\item{} $F_v$ the completion of $F$ at $v.$
\item{} $F_{v}^h$ the henselization of $F$ at $v$ 
($v$  nonarchimedean  if $\text{char} \, F = 0$)
\item{} $\mathcal{O}_v := \left\{
\begin{array}{lll}
\{\alpha \in F_v: \, v(\alpha) \geq 0\} &\rm{if} & v \not | \infty \,\, \text{and char} \, F = 0\\
\{\alpha \in F_v: \, v(\alpha) \geq 0\} &\rm{if} & \text{char} \, F > 0\\
\end{array}\right.$ 
\item{} $k_v :=  \left\{
\begin{array}{lll}
\mathcal{O}_F / v = \mathcal{O}_v /v &\rm{if} & v \notin S_{\infty} \,\, \text{and char} \, F \geq 0\\
\mathcal{O}_v /v  &\rm{if} & v \in S_{\infty} \,\, \text{and char}  \, F > 0\\
\end{array}\right.$ 
\item{} $\overline{F}_s$ \,\, the separable closure of $F.$
\item{} $F_S \subset \overline{F}_s$ the maximal separable extension of $F$ unramified outside $S.$
\item{} $G_{F} := G(\overline{F}_s  / F).$
\item{} $G_{S} := G(F_S / F).$ 
\item{} $W^{n} := W^{n}_{l} := \Q_l / \Z_l (n)$ for any $n \in \Z.$
\item{} $W^{n} (L) := W^{n}_{l} (L) := H^0(G_{L}, \Q_l/\Z_l (n))$ for a field $L$ with $\text{char} \, L \not= l.$ 
\item{} $w_{n} (L) := \prod_{l \not= \text{char} \, L } |W^{n}_{l} (L)|$ whenever 
$|W^{n}_{l} (L)| < \infty$ for every $l \not= \text{char} \, L$ and $|W^{n}_{l} (L)| = 1$ for almost every $l.$ 
\item{} $div A := \{a \in A: \,\,\, \forall_{m \in \Z} \,\,\, \exists_{a^{\prime} \in A}\,\,\, ma^{\prime} = a\}$
for an abelian group $A,$
\item{} $Div A :=$ the maximal divisible subgroup of $A,$
\item{} $A / Div := A / Div A.$
\end{enumerate}

\subsection{Fields of cohomological dimension $\leq$ 2}
Let $L$ be a field. If $L = \F_q$ is a finite field with 
$q$ elements then $\text{cd}_l (\F_q) = 1.$ If $L$ is a local field
it follows from \cite[II, sec. 4.3, Prop. 12,]{Se} that $\text{cd}_l (L) \leq 2.$
If $L = F$ is a global field and $l > 2$ then $\text{cd}_l (F) \leq 2$ by 
\cite[II, sec. 4.3, Prop. 11 and  Prop. 13,]{Se}.
If $\text{char} \, F = 0,$ then
$\text{cd}_2 (F) \leq 2$ iff $F_v = \C$ for every $v | \infty.$ 
It is so because for any $m \geq 3$ and any $2$-torsion, finite $G_F$-module $M$ 
there is the following natural isomorphism \cite[Theorem 4.8 (c) Chap. I]{M1}:  
\begin{equation}
H^m (F, M) \,\, \stackrel{\cong}{\longrightarrow}\,\, \bigoplus_{v \, \text{real}} \,\,
H^m (F_v, M) 
\label{obstacle for cd2 leq 2 for l = 2}\end{equation}
Hence if $\text{char} \, F = 0,$ and $F$ does not have real imbeddings then 
trivially $H^{m} (F_v, M) = 0$ for all $v | \infty,$ all
$G(\overline{F_v} / F_v)$-modules $M$ and all $m > 0.$
This will always be the case in this paper since for $l = 2$ we will assume that 
$\mu_4 \subset F.$ 
\medskip

\noindent
The localization sequence for {\' e}tale cohomology \cite[pp. 267-268]{So1} shows that
$\text{cd}_l (\mathcal{O}_v) \leq 2$ for all nonarchimedean $v$  
and $\text{cd}_l (\mathcal{O}_{F,S}) \leq 2$ for all finite $S \supset S_{\infty, l}$
\medskip

\begin{lemma} Let $L$ be a field such that $\text{char} \, L \not= l$
and $\mu_{l^{\infty}} \subset L.$  
Assume that $K_2 (L^{\prime}) / Div K_2 (L^{\prime})$ is torsion 
for any algebraic extension $L^{\prime}/L.$ Then $\text{cd}_l (L) \leq 1.$
\label{fields cont. all l power roots of 1 have cdl at most 1} 
\end{lemma}
\begin{proof} Let $L^{\prime}/L$ be an algebraic extension.  
Since $\mu_l \subset L^{\prime},$ by Merkurev-Suslin Theorem \cite{MS}: 
\begin{equation}
K_{2} (L^{\prime}) /l K_{2} (L^{\prime}) \stackrel{\cong}{\longrightarrow} 
H^2(L^{\prime}, \, \Z/l (2)) \stackrel{\cong}{\longrightarrow}
Br(L^{\prime})[l] \otimes \Z/l\Z (1). 
\label{Merkurev-Suslin}\end{equation}
By assumption $K_{2} (L^{\prime}) /l K_{2} (L^{\prime}) = K_{2} (L^{\prime})_l /l K_{2} (L^{\prime})_l.$
By Suslin theorem \cite[Theorem 1.8]{Su2} if $\alpha \in K_2 (L^{\prime})[l^k]$ then there is 
$a \in L^{\prime}$ such that $\alpha = \{\xi_{l^k}, \, a\}.$ Hence 
$\alpha$ is divisible by $l$ in $K_{2} (L^{\prime})_l$ because $\mu_{l^{\infty}} \subset L^{\prime}.$
This shows that $Br(L^{\prime})[l] = 0.$ Hence $\text{cd}_l (L) \leq 1$ by \cite[Corollary 2, p. 100]{Sh}.
\end{proof}

\begin{corollary} Let $L$ be an algebraic extension of a global or local field. 
Let $\text{char} \, L \not= l$ and $\mu_{l^{\infty}} \subset L.$ Then $\text{cd}_l (L) \leq 1.$
\label{loc. and glob. fields cont. all l power roots of 1 have cdl at most 1}
\end{corollary}
\begin{proof} 
Let $L$ be an algebraic extension of a global field. 
Observe that for $n > 0$ the $K_{2n}$ groups of rings of integers 
in global fields are finite by results of Borel \cite{Bo}, Harder \cite{Ha} and  Quillen \cite{Q2}. 
Hence by the Quillen localization sequence \cite{Q1} the group $K_2 (L^{\prime})$ is torsion 
for every algebraic extension $L^{\prime}/L.$ 
If $L$ is an algebraic extension of a local field then by \cite{Ta3} and \cite{Me} 
the group $K_2 (L^{\prime})/Div K_2 (L^{\prime})$ is torsion for every algebraic extension 
$L^{\prime}/L.$
Now the claim follows from Lemma \ref{fields cont. all l power roots of 1 have cdl at most 1}.  
\end{proof}

Let $L$ be a field such that $\text{char} \, L \not= l.$
We have
$G(L(\mu_{l^{\infty}})/L) \cong \Delta \times \Gamma$ where 
$\Delta := G(L(\mu_{l})/L)$ and $\Gamma := G(L(\mu_{l^{\infty}})/L(\mu_l)).$ 
 
\begin{lemma} Let $L$ be a field such that $\text{char} \, L \not= l.$ 
If $l =2$ assume that $\mu_{4} \subset L.$ 
Then $\text{cd}_l (G(L(\mu_{l^{\infty}})/L) \leq 1.$
\label{Gal. Gr which has cdl at most 1}
\end{lemma}
\begin{proof} 
By assumptions $\Gamma \cong \Z_l$ if $\mu_{l^{\infty}} \not\subset L$ and 
$\Gamma = 1$ if $\mu_{l^{\infty}} \subset L.$ Moreover
$\Delta \subset {\Z/l}^{\times}$
if $l > 2$ and $\Delta := 1$ if $l=2.$
Consider the spectral sequence for any $l$-torsion 
$G(L(\mu_{l^{\infty}})/L)$-module $M.$ 
\begin{equation}
E_{2}^{p,q} = H^{p} (\Delta, \, H^{q} (\Gamma, \, M)) \Rightarrow H^{p+q} (G(L(\mu_{l^{\infty}})/L),\,  M).
\end{equation}
$E_{2}^{p,q} = 0$ for all $p > 0$ and $q > 1$ because 
$l \not| \, |\Delta|$ and $\text{cd}_l (\Gamma) \leq 1$ by \cite[Chap. IV, Cor. 3.2]{Ri}.   
Hence $H^{m} (G(L(\mu_{l^{\infty}})/L),\,  M) = 0$ for all $m > 1.$
\end{proof}
\medskip

\noindent
The following two theorems are straightforward extensions of well know results of Tate \cite{Ta1}
and Schneider \cite{Sch2} to the framework of general fields.
\begin{theorem} \label{TateIwasawaLetter}
Let $L$ be a field such that $\text{char} \, L \not= l$
and $\mu_{l^{\infty}} \not\subset L.$ If $l =2$ assume that $\mu_{4} \subset L.$ Let
$M$ be a discrete $G(L(\mu_{l^{\infty}})/L)$-module. Then 
\begin{equation}
H^1 (G(L(\mu_{l^{\infty}})/L), \, M \otimes_{\Z} W) = 0.
\label{TateIwasawaLetter1}\end{equation} 
\end{theorem}
\begin{proof} Is clear that $H^1 (\Delta, \, M \otimes_{\Z} W) = 0.$ Hence to get 
(\ref{TateIwasawaLetter1}) it is enough to prove 
$H^1 (\Gamma, \, M \otimes_{\Z} W) = 0$ as in \cite{Ta1} 
and apply the inflation-restriction exact sequence.
\end{proof}

\begin{theorem} \label{Ge of Sch} 
Let $L$ be a field such that $\text{char} \, L \not= l$
and $\mu_{l^{\infty}} \not\subset L.$ If $l =2$ assume that $\mu_{4} \subset L.$
Assume that $K_2 (L^{\prime}) / Div K_2 (L^{\prime})$ is torsion 
for any algebraic extension $L^{\prime}/L(\mu_{l^{\infty}}).$ 
Then
\begin{equation}
H^m(G_{L},\, W^n) \,\, \cong \,\,
\left\{
\begin{array}{lllll}
0 &\rm{if} & m > 2 & \rm{and} & n  \in \Z\\
0&\rm{if}& m = 2 & \rm{and} & n \not= 1\\
Br (L)_l &\rm{if} & m = 2 & \rm{and} & n = 1\\
\end{array}\right.
\label{Ge of Sch1}
\end{equation}
\end{theorem} 

\begin{proof} Consider the spectral sequence:
\begin{equation}
E_{2}^{p,q} = H^{p} (G(L(\mu_{l^{\infty}})/L), \, 
H^{q} (G_{L(\mu_{l^{\infty}})}, \, W^n)) \Rightarrow H^{p+q} (G_L,\,  W^n)
\end{equation}
Observe that 
$E_{2}^{p,q} = 0$ for all $p > 1$ or $q > 1$ by Lemmas \ref{fields cont. all l power roots of 1 have cdl at most 1} 
and \ref{Gal. Gr which has cdl at most 1}. Hence $H^{m} (G_L, \, W^n) = 0$ for all $m > 2$ and all $n \in \Z$ 
and $E_{2}^{2,0} = E_{2}^{0,2} = 0$ for all $n \in \Z.$ If $n \not= 1$ then by Theorem
\ref{TateIwasawaLetter}
$$E_{2}^{1,1} = H^{1} (G(L(\mu_{l^{\infty}})/L), \, H^{1} (G_{L(\mu_{l^{\infty}})}, \, W^n)) =$$
$$ = H^{1} (G(L(\mu_{l^{\infty}})/L), \,\, L(\mu_{l^{\infty}})^{\times} \otimes_{\Z} \, W^{n-1})) = 0.$$   
Hence $H^2(G_{L},\, W^n) = 0$ for $n \not= 1.$ For $n = 1$ the long cohomology exact sequence
associated with the short exact sequence 
$1 \rightarrow \mu_{l^k} \rightarrow \G_m \rightarrow \G_m \rightarrow 1$ and the Hilbert 90
show that $H^2(G_{L},\, \mu_{l^k}) \cong Br(L)[l^k]$ for each $k.$ Hence $H^2(G_{L},\, W) = Br(L)_l.$
\end{proof}

\begin{corollary} Let $L$ be a global or local field with $\text{char} \, L \not= l.$ 
Assume that $\mu_{4} \subset L$ if $l =2.$ Then the isomorphism 
(\ref{Ge of Sch1}) holds for $L.$
\label{Ge of Sch2}
\end{corollary}
\begin{proof} It is shown in the proof of
Corollary \ref{loc. and glob. fields cont. all l power roots of 1 have cdl at most 1}
that $K_2 (L^{\prime}) / Div K_2 (L^{\prime})$ is torsion 
for any algebraic extension $L^{\prime}/L(\mu_{l^{\infty}}).$ Hence the claim
follows from Theorem \ref{Ge of Sch}.  
\end{proof}

\noindent
\subsection{Some useful isomorphisms}
Let $F$ be a global field such that $\mu_4 \subset F$ if $l = 2.$ Then 
the $l$-cohomological dimension of any of 
the following rings ${\mathcal O}_{F,S},$ $F,$ ${\mathcal O}_{v},$ and $F_v$ is $\leq $2 for 
every $l.$ Hence the Dwyer-Friedlander spectral sequence \cite{DF} Proposition 5.1 shows that 
for $X = \, \text{spec} \, {\mathcal O}_{F,S}, 
\, \text{spec} \, F, \, \text{spec} \, {\mathcal O}_{v}, \, \text{spec} \, F_v$ there 
are natural isomorphisms:
\begin{equation}
K_{2n}^{et} (X)  \cong H^{2}_{cont} (X, \Z_l (n+1)),
\quad\quad
K_{2n+1}^{et} (X)  \cong H^{1}_{cont} (X, \Z_l (n+1))
\label{DF SpecSeqResult1}\end{equation} 
\medskip

\noindent
We will often use the following
comparison isomorphisms (\ref{compIso1}),  (\ref{compIso2}),
(\ref{compIso3}) between \' etale cohomology of some affine schemes
and corresponding Galois cohomology. For a commutative ring $R$ with 
identity and an \' etale sheaf $\mathcal{F}$ 
on $\text{spec} \, R$ we put: 
$$H^{\ast} (R, \, \mathcal{F}) :=  
H^{\ast}_{et} (\text{spec} \, R, \, \mathcal{F}).$$
\medskip

\noindent
For a field $K$ and an \' etale sheaf $\mathcal{F}$ on ${\rm{spec}}\, K$
let $M_{\mathcal{F}}$ is the discrete $G_K
:= G(\overline{K}_s/K)$-module corresponding to $\mathcal{F}.$
Then \cite[Chap. II, Theorem 1.9, Chap. III Example 1.7]{M2}: 
\begin{equation}
H^{\ast} ( K, \, \mathcal{F}) \cong  H^{\ast} (G_K, \, M_{\mathcal{F}}),
\label{compIso1}\end{equation}
\bigskip

\noindent
Let $v$ be a nonarchimedean place of a global field $F$ and let $\mathcal{F}$ be an \' etale sheaf  
on $\text{spec} \, {\mathcal O}_{v}.$ Let $G_{v}^{nr}:= G_{F_v} / I_v \cong G(\overline{k_v} / k_v)$
where $I_v$ is the inertia subgroup. Let  $M_{\mathcal{F}}$ be the discrete $G_{v}^{nr}$-module 
corresponding to $\mathcal{F}.$ Then \cite[Chap. III, Th. 4.9]{A}
\begin{equation}
H^{\ast} ({\mathcal O}_{v}, \, \mathcal{F}) \cong  
H^{\ast} (G_{v}^{nr}, \, M_{\mathcal{F}})
\label{compIso2}\end{equation}
\medskip

\noindent
Let $\mathcal{F}$ be a constructible \' etale sheaf on $\text{spec} \, {\mathcal O}_{F,S}$ 
and let $M_{\mathcal{F}}$ be the discrete $G_{F,S}$-module corresponding to $\mathcal{F}.$ Then
\cite[Chap. II, Prop. 2.9]{M1}:
\begin{equation}
H^{\ast} ({\mathcal O}_{F,S}, \, \mathcal{F}) \cong  
H^{\ast} (G_S, \, M_{\mathcal{F}}).
\label{compIso3}\end{equation}
\bigskip

\noindent
From now on in this paper $\text{char} \, F \not= l$ and $\mu_4 \subset F$  if $l = 2.$


\section{Galois cohomology of local and global fields}

\subsection{Tate-Shafarevich groups and Tate-Poitou duality}
\medskip

The r-th Tate-Shafarevich group $\Sha^{r}_{S} (F, M)$ for a $G_{F, S}$-module 
$M$ is defined as follows \cite[p. 70]{M1}:  

\begin{equation}
\Sha^{r}_{S} (F, M) \,\, : = \,\, \text{ker}  \, ( H^r (G_{F, S}, \, M) \,\,  
\stackrel{}{\longrightarrow}\,\,  
\,\,  \prod_{v \in S} \, H^r (G_{F_v}, \, M)) 
\label{S-Sha}\end{equation}
In this paper $S$ is finite. In loc. cit. $\Sha^{r}_{S} (F, M)$ 
is defined for any nonempty $S$ (containing $S_{\infty}$  if  $\text{char}\, F = 0$).
In particular for a $G_{F}$-module $M:$  
\begin{equation}
\Sha^{r} (F, M) \,\, : = \,\, \text{ker} \, ( H^r (G_F,\, M) \,\,  
\stackrel{}{\longrightarrow}\,\,  
\,\,  \prod_v \, H^r (G_{F_v}, \, M)) 
\label{Sha}\end{equation}

\noindent
Observe that for an abelian variety $A/F$ and $M = A(\overline{F}_s)$ we have 
$\Sha^{1} (F, A(\overline{F})) \subset \Sha (A/F)$ where 
$\Sha (A/F)$ is the classical Tate-Shafarevich group. 
By Tate-Poitou duality (see eg. \cite[Theorem 4.10, Chap. I]{M1}) for any finite $G_{S}$-module $M$ 
with order being a unit in ${\mathcal O}_{F,S}$ and for the $G_{S}$-module
$M^{D} := Hom (M, \,\overline{F}^{\times})$ there is the following perfect pairing
\begin{equation}
\Sha^{1}_{S} (F, M) \times \Sha^{2}_{S} (F, M^{D}) 
\stackrel{}{\longrightarrow}\,\, \Q/\Z
\label{GeneralTateDuality}\end{equation} 

\noindent
Since $\Z/ l^k (n)^D \cong \Z/ l^k (1-n)$ we get perfect pairing:
\begin{equation}
\Sha^{1}_{S} (F, \Z/ l^k (n)) \times \Sha^{2}_{S} (F, \Z/ l^k (1-n)) 
\stackrel{}{\longrightarrow}\,\, \Z/l^k
\label{TateDuality2}\end{equation} 
Passing on the left and the target of (\ref{TateDuality2}) to the direct limit and 
on the right of (\ref{TateDuality2}) to the inverse limit we get perfect pairing: 
\begin{equation}
\Sha^{1}_{S} (F, \Q_l/ \Z_l (n)) \times \Sha^{2}_{S} (F, \Z_l (1-n)) 
\stackrel{}{\longrightarrow}\,\, \Q_l / \Z_l
\label{TateDuality3}\end{equation} 

Let $M$ be a finite 
discrete $G({\overline F} / F)$-module. 
If $\rho_M \, :\ G({\overline F} / F) \rightarrow Aut_{\Z} (M),$ then we put
$F(M) := {\overline F}^{\text{ker} \, \rho_M}.$ Let $v$ be a place of $F$ and 
let $w$ denote a place of $F (M)$ over $v.$ 
Consider the following commutative diagram with exact columns c.f. \cite[p. 79]{Ne}:
$$
\small{\xymatrix{
& 0 \ar@<0.1ex>[d]^{} & 0 \ar@<0.1ex>[d]^{}\\
& \quad H^1 (G(F( M)/ F), M)   
\ar@<0.1ex>[d]^{} \ar[r]^{}  \quad & 
\quad  \prod_v \prod_{w | v}\, H^1 (G(F( M)_w/ F_v), M) \ar@<0.1ex>[d]^{} \\
\quad & 
\quad H^1 (G_F,\, M)  \ar@<0.1ex>[d]^{} \ar[r]^{}  \quad & 
\quad  \prod_v \prod_{w | v}\, H^1 (G_{F_v}, \, M) \ar@<0.1ex>[d]^{} \\
0 \quad \ar[r]^{}  \quad & 
\quad H^1 (G_{F( M)}, \, M)   
\ar[r]^{}  \quad & 
\quad \prod_{v} \prod_{w | v} \, H^1 (G_{F( M)_w}, \, M) \\}}
\label{}$$

\noindent
By Chebotarev's density theorem the bottom horizontal 
arrow is a monomorphism, hence 
$\Sha^{1} (F (M), M) = 0.$   
Hence it is clear that the upper horizontal arrow is a monomorphism if 
and only if $\Sha^{1} (F, M) = 0.$ 
If $F (M) / F$ is cyclic then by Chebotarev's theorem there are infinitely many
places $v$ such that $G(F( M)/ F) = G(F(M)_w/ F_v).$ So it is clear that in 
this case the top horizontal arrow is a monomorphism hence 
$\Sha^{1} (F, M) = 0.$ 
\medskip

\noindent
The most interesting case for this paper is 
$M = \Z/ l^k (n),$ where $l \not= \text{char}\, F$ and $\mu_4 \in F$ if $l = 2.$ 
In this case all horizontal arrows in the above diagram are monomorphisms.
Hence for every $n \in \Z$ and every $k \geq 0$  
\begin{equation}
\Sha^{1} (F, \Z/ l^k (n)) = 0.
\label{Sha1 is zero for Z mod lk}
\end{equation}
By Tate-Poitou duality (\ref{TateDuality2}) for every $n \in \Z$ and every $k \geq 0$ 
\begin{equation}
\Sha^{2} (F, \Z/ l^k (n)) = 0.
\label{Sha2 is zero for Z mod lk}
\end{equation}
The equalities (\ref{Sha1 is zero for Z mod lk}) and 
(\ref{Sha2 is zero for Z mod lk}), in the number field case, 
have already been observed by Neukirch \cite[Satz 4.5]{Ne}. Hence (\ref{Sha2 is zero for Z mod lk}) 
and the exact sequence of $G_F$-modules
\begin{equation}
0 \stackrel{}{\longrightarrow} \Z/l^k (n) \stackrel{}{\longrightarrow} W^n \stackrel{l^k}{\longrightarrow}
W^n \stackrel{}{\longrightarrow} 0
\label{QlZl times l QlZl exact sequence}\end{equation}
give the following commutative diagram with exact rows and columns:
$$\small{\xymatrix{
& 0 \ar@<0.1ex>[d]^{} & 0 \ar@<0.1ex>[d]^{}\\
0 \quad \ar[r]^{} \quad 
& \quad H^1 (G_F, W^n) / l^k   
\ar@<0.1ex>[d]^{} \ar[r]^{\prod_v \, r_v}  \quad & 
\quad  \prod_v \, H^1 (G_{F_v}, W^n)/l^k \ar@<0.1ex>[d]^{} \\
0 \quad \ar[r]^{} \quad & 
\quad H^2 (G_F, \Z/l^k (n))  \ar[r]^{\prod_v \, r_v}  \quad & 
\quad  \prod_v \, H^2 (G_{F_v}, \, \Z/l^k (n))  \\}}
\label{}$$

\noindent
This diagram gives the following equality: 
\begin{equation}
div \,  H^1 (G_F, W^n) \, = \, \{h: \,\, 
r_v (h) \in div ( H^1 (G_{F_v}, W^n)), \,\, \text{for all} \,\, v \}
\label{Schneider Lemma 4.4}
\end{equation}
For an abelian group $M$ put $M^{\ast} := Hom (M, \, \Q/ \Z).$ For a finite $S$ containing
$S_{\infty, l}$ and for a finite $G_S$-module $M$ Tate-Poitou duality gives the following exact sequence.
\begin{equation}\small{
0 \rightarrow H^0 (G_{S}, M) \stackrel{}{\longrightarrow}
\bigoplus_{v \in S} H^0 (G_{F_v}, M) \stackrel{}{\longrightarrow} H^2 (G_{S}, M^{D})^{\ast}}
\label{longexactTatePoitou1}\end{equation}
\begin{equation}
\small{\rightarrow  H^1 (G_{S}, M) \stackrel{}{\longrightarrow}
\bigoplus_{v \in S} H^1 (G_{F_v}, M) \stackrel{}{\longrightarrow}  H^1 (G_{S}, M^{D})^{\ast}}
\nonumber\end{equation}
\begin{equation}\small{
\rightarrow H^2 (G_{S}, M) \stackrel{}{\longrightarrow}
\bigoplus_{v \in S} H^2 (G_{F_v}, M) \stackrel{}{\longrightarrow} H^0 (G_{S}, M^D)^{\ast}
\rightarrow 0}
\nonumber\end{equation}
where for $v$ archimedean $H^i (G_{F_v}, M)$ is $H^i_{T} (G_{F_v}, M).$
Taking $M := \Z/l^k (n)$ and passing to direct limits gives the exact sequence:
\begin{equation}\small{
0 \rightarrow H^0 (G_{S}, W^n) \stackrel{}{\longrightarrow}
\bigoplus_{v \in S} H^0 (G_{F_v}, W^n) \stackrel{}{\longrightarrow} H^2 (G_{S}, \Z_l (1-n))^{\ast}}
\label{longexactTatePoitou2}\end{equation}
\begin{equation}\small{
\rightarrow H^1 (G_{S}, W^n) \stackrel{}{\longrightarrow}
\bigoplus_{v \in S} H^1 (G_{F_v}, W^n) \stackrel{}{\longrightarrow} H^1 (G_{S}, \Z_l (1-n))^{\ast}}
\nonumber\end{equation}
\begin{equation}\small{
\rightarrow H^2 (G_{S}, W^n) \stackrel{}{\longrightarrow}
\bigoplus_{v \in S} H^2 (G_{F_v}, W^n) \stackrel{}{\longrightarrow} H^0 (G_{S}, \Z_l (1-n))^{\ast}
\rightarrow 0}
\nonumber\end{equation}
We notice that for $v$ archimedean $H^i_{T} (G_{F_v}, M) = 0$ for an $l$-torsion $G_{F_v}$-module $M$ since $G_{F_v} = G(\C/\R)$ or trivial and $\mu_{4} \subset F$
if $l = 2.$ Hence in our case there is no contribution of archimedean part 
to the Tate-Poitou exact sequences with $l$-torsion Galois modules.


\subsection{Divisible elements in Galois cohomology of local fields}
For every prime $v$ (nonarchimedean if ${\rm char} \, F = 0$)
the Theorem \ref{Ge of Sch} and Corollary \ref{Ge of Sch2} give:
\begin{equation}
H^m(G_{F_v},\, W^n)
\,\, \cong \,\,
\left\{
\begin{array}{lllll}
0 &\rm{if} & m > 2 & \rm{and} & n  \in \Z\\
0&\rm{if}& m = 2 & \rm{and} & n \not= 1\\
Br (F_v)_l \cong \Q_l/\Z_l  & \rm{if} & m = 2 & \rm{and} & n = 1\\
\end{array}\right.
\label{generallocalfieldSchneider111}
\end{equation}

\noindent
It follows from local Tate duality \cite[I Corollary 2.3]{M1} or 
\cite[II, sec. 5.2 Theorem 2 and the remark following it]{Se},  
that for each $i$ such that $0 \leq i \leq 2$ there is 
a perfect pairing
\begin{equation}
H^{i} (F_v, \Q_l / \Z_l (n)) \times H^{2-i} (F_v, \Z_l (1-n)) \rightarrow \Q_l /\Z_l.
\label{LocalTate Duality}\end{equation}
Hence the following group is finite for every $n \not= 1:$
\begin{equation}
H^2(G_{F_v}, \Z_l (n)) \cong H^0(G_{F_v}, W^{1-n})^{\ast}  \cong 
W^{n-1} (F_v),
\label{generallocalfieldSchneider112}
\end{equation}

\noindent
so by (\ref{generallocalfieldSchneider111}):
\begin{equation}
H^1(G_{F_v},\, W^n)/Div  \cong H^2(G_{F_v}, \Z_l (n)) 
\,\,\, \text{for} \,\,\, n \not= 1.
\label{generallocalfieldSchneider12}
\end{equation}

\noindent
Moreover by Hilbert 90 we have:
\begin{equation}
H^1(G_{F_v}, W^1) \cong F_{v}^{\times} \otimes \Q_l/ \Z_l \cong Div \, H^1(G_{F_v}, \, W^{1}) .
\label{generallocalfieldSchneider13}
\end{equation}

\noindent
Hence for any local field $F_v,$ any prime $l \not = {\rm char} \, F_v$ and any 
$n \in \Z:$ 
\begin{equation}
div \, H^1(G_{F_v},\, W^n) = Div \, H^1(G_{F_v},\, W^n) .
\label{generallocalfieldSchneider13}
\end{equation}

\noindent
By (\ref{compIso2}), (\ref{generallocalfieldSchneider112}) and (\ref{generallocalfieldSchneider12}),
for any $v \notin S_l,$ the boundary map $\partial_v$ in the localization sequence 
for $\mathcal{O}_v$ gives the following isomorphism:
\begin{equation}
\partial_v \,\, :\,\, H^1(G_{F_v}, W^n)/ Div \stackrel{\cong}{\longrightarrow}
H^{0} (G(\overline{k_v}/k_v), \, W^{n-1}) 
\label{generallocalfieldSchneider113}
\end{equation}
\medskip

\begin{theorem}\label{generallocalfieldSchneider2} 
There is the following isomorphism:
\begin{equation}
Div \, H^1(G_{F_v}, W^n)
\,\, \cong \,\,
\left\{
\begin{array}{lll}
(\Q_l/\Z_l)^{[F_v:\, \Q_l] + 1} &\rm{if} &  v \in S_l, \,\,\, n \in  \{0, 1\}\\
(\Q_l/\Z_l)^{[F_v:\, \Q_l]} &\rm{if} & v \in S_l, \,\,\, n \notin \{0, 1\}\\
\Q_l/\Z_l&\rm{if} & v \notin S_l, \,\,\, n \in  \{0, 1\}\\
0&\rm{if} & v \notin S_l, \,\,\, n \notin  \{0, 1\}\\
\end{array}\right.
\label{generallocalfieldSchneider14}\end{equation}
\end{theorem}
\begin{proof}
For any finite $G_{F_v}$-module $M$ with order $m := |M|$ prime to $\text{char} \, F_v$ 
the Euler characteristic can be computed as follows \cite[chap. I Theorem 2.8 ]{M1}:
\begin{equation}
\chi (G_{F_v}, M) := \frac{|H^0 (G_{F_v}, \, M)| \, |H^2 (G_{F_v}, \, M)| }{|H^1 (G_{F_v}, \, M)|}
\, = \, |\,\, {\mathcal O}_v / m {\mathcal O}_v \,\, |^{-1} 
\label{Euler char 1}\end{equation}
Hence 
\begin{equation}
\chi (G_{F_v}, \Z/l (n)) =
\left\{
\begin{array}{lll} 
l^{- [F_v : \, \Q_l]} &\rm{if} & v \in S_l\\
1 &\rm{if} & v \notin S_l
\end{array}\right.
\label{Euler char 2}\end{equation}
Taking $log_l$ in (\ref{Euler char 2}) gives
\begin{equation}
\sum_{i=0}^2 {\rm{dim}}_{\Z/l}\,  H^i (G_{F_v}, \Z/l (n)) \,\, = 
\left\{
\begin{array}{lll} 
-[F_v : \, \Q_l] &\rm{if} & v \in S_l\\
0 &\rm{if} & v \notin S_l
\end{array}\right.
\label{Euler char 3}
\end{equation}
Observe that 
\begin{equation}
H^0 (G_{F_v}, W^n) \cong
\left\{
\begin{array}{lll}
\Q_l/\Z_l &\rm{if} & \,\,\, n = 0\\
{\rm{finite}}&\rm{if} &  \,\,\, n \not= 0\\
\end{array}\right.
\label{H^0(Ql mod Zl for small n)}
\end{equation}

\noindent
Computing the divisible rank of $H^1 (G_{F_v}, W^n)$ by use of (\ref{generallocalfieldSchneider111}), 
(\ref{Euler char 3}), (\ref{H^0(Ql mod Zl for small n)}) and the following exact 
sequence gives the formula (\ref{generallocalfieldSchneider14}). 
\begin{equation}
0 \rightarrow H^0 (G_{F_v}, \Z/l (n)) \stackrel{}{\longrightarrow}
H^0 (G_{F_v}, W^n) \stackrel{l}{\longrightarrow} H^0 (G_{F_v}, W^n)
\label{longexactH0H1H2}\end{equation}
\begin{equation}
\rightarrow H^1 (G_{F_v}, \Z/l (n)) \stackrel{}{\longrightarrow}
H^1 (G_{F_v}, W^n) \stackrel{l}{\longrightarrow} H^1 (G_{F_v}, W^n)
\nonumber\end{equation}
\begin{equation}
\rightarrow H^2 (G_{F_v}, \Z/l (n)) \stackrel{}{\longrightarrow}
H^2 (G_{F_v}, W^n) \stackrel{l}{\longrightarrow} H^2 (G_{F_v}, W^n) \rightarrow 0.
\nonumber\end{equation}
\end{proof}

\begin{remark}
Theorem \ref{generallocalfieldSchneider2} was proved by P. Schneider \cite[Satz 4 sec. 3]{Sch2}
for $l > 2$ and $\text{char} \, F_v = 0.$
\end{remark}


\subsection{Divisible elements in Galois cohomology of global fields}
Consider any first quadrant spectral sequence $E_{2}^{i,j} \Rightarrow E^{n}$ with 
$n = i + j \geq 0.$ The differentials $d_{r}^{i,j} \, :\, E_{r}^{i,j} \rightarrow
E_{r}^{i+r,j-r+1}$ define
$E_{r+1}^{i,j} := \text{ker} \, d_{r}^{i,j} \, / \, \text{im} \, d_{r}^{i-r,j+r-1}$ 
for every $r \geq 2.$ 
For $r > n+1$ we have $d_{r}^{i,j} = 0.$ Put 
$E_{\infty}^{i,j} := E_{n+2}^{i,j} = E_{n+3}^{i,j} = \dots.$ 
The filtration $0 \subset E^{n}_{n} \subset 
E^{n}_{n-1} \subset \dots \subset E^{n}_{0} = E^n$ gives 
$E^{n}_{i} / E^{n}_{i+1} \cong E^{i, j}_{\infty}.$   
If $E_{2}^{i,j} = 0$ for every $i > 2$ then the exact sequence 
of lower terms extends to the following exact sequence: 
\begin{equation}
0 \rightarrow E_{2}^{1,0} \rightarrow E^{1} \rightarrow E_{2}^{0,1} \rightarrow E_{2}^{2,0} \rightarrow
E_{1}^{2} \rightarrow E_{2}^{1,1} \rightarrow 0
\label{SixTermExactSSeq.oflowerTerms}\end{equation}
\noindent
Consider the Leray spectral sequence for the natural map $j \, : \, \text{spec} \, F 
\rightarrow \, \text{spec} \,{\mathcal O}_{F,S}:$
\begin{equation}
E_{2}^{i,j} = H^i ({\mathcal O}_{F,S}, \, R^j j_{\ast} W^n) \Rightarrow H^{i+j} 
(F, \, W^n) 
\label{OS to F spectral sequence}\end{equation}
Since $\text{cd}_l (G_S) = \text{cd}_l ({\mathcal O}_{F,S}) = 2$ the exact sequence 
(\ref{SixTermExactSSeq.oflowerTerms}) gives the following localization 
exact sequence in cohomology:
\begin{equation}\small{
0 \rightarrow H^1 ({\mathcal O}_{F,S}, W^n) \stackrel{}{\longrightarrow}
 H^1 (F, W^n) \stackrel{\partial}{\longrightarrow} \bigoplus_{v \notin S} H^0 (k_v, W^{n-1})
\rightarrow}
\label{localization seq. in coh. 1}\end{equation}
\begin{equation}\small{
\rightarrow H^2 ({\mathcal O}_{F,S}, W^n) \stackrel{}{\longrightarrow}
 H^2 (F, W^n) \stackrel{\partial}{\longrightarrow} \bigoplus_{v \notin S} H^1 (k_v, W^{n-1})
\rightarrow 0}
\nonumber\end{equation}
By Theorem \ref{Ge of Sch} and Corollary \ref{Ge of Sch2} we get
$H^2 (G_F, W^n) = 0$ for all $n \not= 1.$ Hence for any $n \not= 1$ the sequence 
(\ref{localization seq. in coh. 1}) has the following form:
\begin{equation}\small{
0 \rightarrow H^1 ({\mathcal O}_{F,S}, W^n) \stackrel{}{\longrightarrow}
 H^1 (F, W^n) \stackrel{\partial}{\longrightarrow} \bigoplus_{v \notin S} H^0 (k_v, W^{n-1})
\rightarrow H^2 ({\mathcal O}_{F,S}, W^n)
\rightarrow 0}
\label{localization seq. in coh. 2}\end{equation}

\noindent
By \cite[Prop. 2.3]{Ta2} there is the following exact sequence:
\begin{equation}
\small{0 \rightarrow H^{1} (G_{S}, \, W^n) / Div 
\rightarrow H^{2} (G_{S}, \, \Z_l (n))
\rightarrow H^{2} (G_{S}, \, \Q_l (n)) \rightarrow H^{2} (G_{S}, \, W^n) \rightarrow 0}
\label{TateExact1}\end{equation}
P. Schneider defined numbers $i_n (F)$ as follows:
\begin{equation}
i_n (F) := \text{dim}_{\F_l} (Div H^{2} (G_{S}, \, W^n))[l]\, 
\nonumber\end{equation}
In other words $i_n (F)$ is the number of copies of
$\Q_l/\Z_l$ in $Div \, H^{2} (G_{S}, \, W^n).$ Note that
$i_n (F) = 0$ iff  $H^{2} (G_{S}, \, \Q_l (n)) = 0$ since 
$H^{i} (G_{S}, \, \Z_l (n))$ are finitely generated $\Z_l$-modules.  
It is immediate from (\ref{localization seq. in coh. 2}) that $i_n (F)$ does not depend on the finite 
set $S$ and for every finite $S$ containing $S_{l}:$
\begin{equation}
Div \, H^1 (G_{F, S}, W^n) = Div \, H^1 (G_{F}, W^n).
\label{MaxDivDoesNotDependOnS}
\end{equation}

\begin{conjecture} (P. Schneider) \quad
$i_n = 0$ for all $n \not= 1.$
\end{conjecture}

The following lemma is well known. It was first proven by 
Soul{\' e} \cite[Th{\' e}or{\` e}me 5]{So1} in the number 
field case for $l > n.$ We include here a proof that works 
for all global fields and $l \geq 2$ provided $\mu_4 \subset F$ if $l = 2.$ 

\begin{lemma} $H^2 ({\mathcal O}_{F,S}, W^n) = 0$ for any $n > 1$ and 
any finite $S \supset S_{\infty, l}.$ In particular 
$i_n = 0$ for all $n > 1.$
\label{SchneiderConj.}\end{lemma}

\begin{proof} 
Consider the following commutative diagram for each $n > 1:$
$$\small{\xymatrix{
0 \ar[r]^{}  & K_{2n -2} ({\mathcal O}_{F,S})_l \ar@<0.1ex>[d]^{} \ar[r]^{}  
&  K_{2n-2} (F)_l \ar@<0.1ex>[d]^{} \ar[r]^{\partial} & 
 \bigoplus_{v \notin S} \, K_{2n-3} (k_v)_l \ar@<0.1ex>[d]^{\cong} \ar[r]^{} & 0 & \\
0\ar[r]^{} & H^{1} ({\mathcal O}_{F,S}, \, W^n) / Div  \ar[r]^{}  & H^{1} (F, \, W^n) / Div 
\ar[r]^{\partial} & \bigoplus_{v \notin S} \,  H^{0} (k_{v}, \, W^{n-1}) \ar[r]^{}  & 0}}
\label{Com.Diag.Kth.Coh1} $$

The top row is exact by Quillen localization sequence \cite{Q1} and results of Soul\' e
\cite[Th\' eor\`eme 3 p. 274]{So1}, \cite[Th\' eor\`eme 1 p. 326]{So2}. 
The left and the middle vertical arrows are surjective by
\cite[Theorems 8.7 and 8.9]{DF} and the right vertical arrow is an isomorphism 
by \cite[Corollary 8.6]{DF}. This implies that the bottom sequence is also exact.
Hence (\ref{localization seq. in coh. 2}) shows that
$ H^2 ({\mathcal O}_{F,S}, W^n) = 0$ for all $n > 1.$ So $i_n = 0$
for all $n > 1.$ 
\end{proof}

Put:
\begin{equation}
D_n (F) := div ( H^1 (G_{F}, W^n) / Div)
\nonumber\end{equation} 
\medskip

\noindent
The following theorem extends \cite[Satz 8 sec. 4]{Sch2}.
\begin{theorem}\label{functionfieldSchneider3}
Assume that $i_n = 0$ for $n \not= 1.$ 
There are the following exact sequences:
\begin{equation}
\small{0 \rightarrow D_{n} (F) \rightarrow
H^{1} (G_{S}, \, W^n) / Div 
\rightarrow \bigoplus_{v \in S} \, W^{n-1} (F_v) 
\rightarrow W^{n-1} (F) \rightarrow
0.}
\label{Gen.Sch.OS}\end{equation}

\begin{equation}
\small{0 \rightarrow D_{n} (F) \rightarrow
H^{1} (G_F, \, W^n) / Div 
\rightarrow \bigoplus_{v} \, W^{n-1} (F_v) 
 \rightarrow W^{n-1} (F) \rightarrow 0} 
\label{Gen.Sch.F}\end{equation}
\end{theorem}

\begin{proof} Let us prove the exactness of (\ref{Gen.Sch.OS}).
Substituting $n$ for $1-n$ in the first three terms of the exact sequence (\ref{longexactTatePoitou2}),
dualizing  and applying \cite[prop. 2.3]{Ta2} gives us the following exact sequence: 
\begin{equation}
\small{
H^{1} (G_{S}, \, W^n) / Div 
\rightarrow \bigoplus_{v \in S} \, H^{1} (G_{F_v}, \, W^n) / Div
\rightarrow W^{n-1} (F) \rightarrow
0.}
\label{Gen.Sch.OS1}\end{equation}

\noindent
For every $S \supset S_{\infty, l}$ consider the following commutative diagram: 
$$\small{\xymatrix{
0 \ar[r]^{} & D_n (F) \ar@<0.1ex>[d]^{} \ar[r]^{} & 
 H^{1} (G_{F}, \, W^n) / Div  \ar@<0.1ex>[d]^{\cong} \ar[r]^{\partial} & 
\bigoplus_{v} \, H^{1} (G_{F_v}, \, W^n) / Div \ar@<0.1ex>[d]^{}& \\
0 \ar[r]^{} & H^{1} ({\mathcal O}_{F,S}, \, W^n) / Div  \ar[r]^{} & H^{1} (F, \, W^n) / Div 
\ar[r]^{\partial} & \bigoplus_{v \notin S} \,  H^{0} (k_{v}, \, W^{n-1}) & }}
\label{Com.Diag.Coh.Coh1} $$
The exactness of the top sequence follows by (\ref{Schneider Lemma 4.4}) and 
(\ref{generallocalfieldSchneider13}). By (\ref{generallocalfieldSchneider113}) this gives the 
following commutative diagram:
$$\small{\xymatrix{
0 \ar[r]^{} & D_n (F) \ar[r]^{}  & 
 H^{1} (G_{F}, \, W^n) / Div   \ar[r]^{\partial} & 
\bigoplus_{v} \, H^{1} (G_{F_v}, \, W^n) / Div & \\
0 \ar[r]^{} & D_n (F) \ar[r]^{} \ar@<0.1ex>[u]^{=} & 
H^{1} ({\mathcal O}_{F,S}, \, W^n) / Div  \ar[r]^{} \ar@<0.1ex>[u]^{} & 
\bigoplus_{v \in S} \, H^{1} (F_v, \, W^n) / Div \ar@<0.1ex>[u]^{}  & }}
\label{Com.Diag.Coh.Coh2} $$
Hence the exact sequence (\ref{Gen.Sch.OS}) is obtained by connecting 
the bottom exact sequence of the last diagram and the exact sequence
(\ref{Gen.Sch.OS1}) and applying isomorphisms 
(\ref{generallocalfieldSchneider112}) and (\ref{generallocalfieldSchneider12}). 
The exact sequence (\ref{Gen.Sch.F}) is obtained from (\ref{Gen.Sch.OS})
by passing to the direct limit over $S.$
\end{proof}

\begin{corollary}\label{DIVandSZA} Let $n \not= 1$ and let
$i_{n} = 0.$ For every finite $S \supset S_{\infty, l}:$ 
\begin{equation}
D_{n} (F) = \Sha^{2}_{S} (F,\, \Z_l (n)) = \Sha^{2} (F,\, \Z_l (n))_l.
\label{DIVandSZA1}\end{equation}
\end{corollary}
\begin{proof} Follows by Theorem (\ref{functionfieldSchneider3})
and \cite[Prop. 2.3]{Ta2}. 
\end{proof}

\noindent
Let (see (\ref{notation D(n)}) in the next chapter): 
\begin{equation}
D^{et}(n) := div \, K_{2n}^{et}(F)_l.
\end{equation}
\begin{theorem}\label{MooreEt} Let $n > 0.$ 
For every finite $S \supset S_{\infty, l}$ there are exact sequences: 
\begin{equation}
\small{0 \rightarrow D^{et}(n) \rightarrow
K_{2n}^{et}({\mathcal O}_{F, S})  
\rightarrow \bigoplus_{v \in S} \, W^{n} (F_v) 
\rightarrow W^{n} (F) \rightarrow
0.}
\label{MooreEt OS}\end{equation}
\begin{equation}
\small{0 \rightarrow D^{et}(n) \rightarrow
K_{2n}^{et}(F)_l 
\rightarrow \bigoplus_{v} \, W^{n} (F_v) 
 \rightarrow W^{n} (F) \rightarrow 0} 
\label{MooreEt F}\end{equation}
Moreover there is the following equality: 
\begin{equation}
\frac{|K_{2n}^{et}({\mathcal O}_{F, S})|}{| D^{et}(n)|} = 
\frac{{\bigl |}\prod_{v \in S} w_{n} (F_v) {\bigl |}_{l}^{-1} }{|w_{n} (F)|_{l}^{-1}}.
\label{EtKTh OS comp to Div}\end{equation}
\end{theorem}

\begin{proof} 
By Lemma \ref{SchneiderConj.} the sequences (\ref{Gen.Sch.OS}) and 
(\ref{Gen.Sch.F}) are exact. Moreover by \cite[Prop. 5.1]{DF} and 
\cite[prop. 2.3]{Ta2} there are the following isomorphisms:
\begin{equation}
K_{2n}^{et} ({\mathcal O}_{F,S}) \cong H^{2} (G_{F, S}, \, \Z_l (n+1))_l \cong 
H^{1} (G_{F, S}, \, W^{n+1}) / Div.
\nonumber\end{equation}
\begin{equation} 
K_{2n}^{et} (F)_l \cong H^{2} (G_{F}, \, \Z_l (n+1))_l \cong H^{1} (G_{F}, \, W^{n+1}) / Div
\nonumber\end{equation}
Hence  $D^{et}(n) \cong D_{n+1} (F).$
Observe that the group $H^{1} (G_{F, S}, \, W^{n+1}) / Div$ is finite. 
Indeed by \cite{Bo}, \cite{Ha} and \cite{Q2} the group 
$K_{2n} ({\mathcal O}_{F,S})$ is finite and the Dwyer-Friedlander map
$K_{2n} ({\mathcal O}_{F,S}) \rightarrow K_{2n}^{et} ({\mathcal O}_{F,S})$ \cite{DF}
is surjective. The formula (\ref{EtKTh OS comp to Div}) follows from (\ref{MooreEt OS}) because all the terms of  
this exact sequence are finite. 
\end{proof}


\section{Divisible elements in K-groups of global fields}
\subsection{General results on divisible elements}

\noindent
Consider the following commutative diagram. The rows are localization 
sequences and the vertical maps are the Dwyer-Friedlander maps \cite{DF}.


$$\small{\xymatrix{
\ar[r]^{}  & K_{2n+1} (F,\, \Z/l^k)   
\ar@<0.1ex>[d]^{} \ar[r]^{\partial}  & 
 \bigoplus_{v} \, K_{2n} (k_v, \, \Z/l^k) \ar@<0.1ex>[d]^{\cong} \ar[r]^{} & 
K_{2n} ({\mathcal O}_F, \, \Z/l^k) \ar[r]^{} \ar@<0.1ex>[d]^{} &  \\
\ar[r]^{} & K_{2n+1}^{et} (F,\, \Z/l^k)  \ar[r]^{\partial^{et}}  & 
 \bigoplus_{v \not \, | \, l} \, K_{2n}^{et} (k_v, \, \Z/l^k)  \ar[r]^{} & 
  K_{2n}^{et} ({\mathcal O}_{F}[\frac{1}{l}], \, \Z/l^k) \ar[r]^{} &}}
\label{CoeffLocSeq} $$
For every $k > 0$ define: 
\begin{equation}
D (n, l^k) := 
\text{ker} \,  (K_{2n} ({\mathcal O}_F, \,  \Z /l^k) 
\rightarrow K_{2n} (F, \,  \Z/l^k)) \, = \, \text{coker}\, \partial  
\nonumber \end{equation} 
\begin{equation}
D^{et} (n, l^k) := 
\text{ker} \,  (K_{2n}^{et} ({\mathcal O}_{F} [1/l], \,  \Z /l^k) 
\rightarrow K_{2n}^{et} (F, \,  \Z/l^k)) \, = \, \text{coker}\,  \partial^{et} 
\nonumber \end{equation} 
\medskip

\noindent
We do not consider $\text{coker}\, \partial$ and 
$\text{coker} \,  \partial^{et}$ in the following commutative diagram: 
$$\small{\xymatrix{
\ar[r]^{}  & K_{2n} (F,\, \Z/l^k)   
\ar@<0.1ex>[d]^{} \ar[r]^{\partial}  & 
 \bigoplus_{v} \, K_{2n-1} (k_v, \, \Z/l^k) \ar@<0.1ex>[d]^{\cong} \ar[r]^{} & 
K_{2n-1} ({\mathcal O}_F, \, \Z/l^k) \ar[r]^{} \ar@<0.1ex>[d]^{} &  \\
\ar[r]^{} & K_{2n}^{et} (F,\, \Z/l^k)  \ar[r]^{\partial^{et}}  & 
 \bigoplus_{v \not \, | \, l} \, K_{2n-1}^{et} (k_v, \, \Z/l^k)  \ar[r]^{} & 
  K_{2n-1}^{et} ({\mathcal O}_{F}[\frac{1}{l}], \, \Z/l^k) \ar[r]^{} &}}
\label{CoeffLocSeqNoInterest} $$
because the isomorphism $K_{2n-1} ({\mathcal O}_F) 
\cong K_{2n-1} (F)$ (resp. the isomorphism 
$K_{2n-1}^{et} ({\mathcal O}_F [\frac{1}{l}])_l \cong K_{2n-1}^{et} (F)_l$) for every $n > 1$ and 
the comparison of Bockstein sequences for ${\mathcal O}_F$ and $F$ 
(resp. ${\mathcal O}_F [\frac{1}{l}]$ and $F$) show that: 
\begin{equation}
\text{coker}\, \partial \,\, = \,\, \text{ker} \,  (K_{2n-1} ({\mathcal O}_F, \,  \Z /l^k) 
\rightarrow K_{2n-1} (F, \,  \Z/l^k)) \,\,   = 0  
\nonumber \end{equation} 
\begin{equation} 
\text{coker}\,  \partial^{et} \,\, = \,\, \text{ker} \,  (K_{2n-1}^{et} ({\mathcal O}_{F} [1/l], \,  \Z /l^k) 
\rightarrow K_{2n-1}^{et} (F, \,  \Z/l^k)) \,\, = 0
\nonumber \end{equation} 
Comparing the Bockstein exact sequences in K-theory for ${\mathcal O}_F$
and for $F$ (resp. {\' e}tale K-theory for ${\mathcal O}_{F}[1/l]$
and for $F$) we notice that for each $k > 0$: 

\begin{equation}
\small{D (n, l^k) \cong \, \text{ker} \,  (K_{2n} ({\mathcal O}_F) / l^k 
\rightarrow K_{2n} (F) / l^k) \cong}
\label{1.1}\end{equation}
\begin{equation}\small{
\cong K_{2n} ({\mathcal O}_F) \cap 
K_{2n} (F)^{l^k} / K_{2n} ({\mathcal O}_F)^{l^k}.}
\nonumber\end{equation}

\begin{equation}\small{
D^{et} (n, l^k) \cong \, \text{ker} \,  (K_{2n}^{et} ({\mathcal O}_{F}[1/l]) / l^k 
\rightarrow K_{2n}^{et} (F) / l^k) \cong}
\label{1.11}\end{equation}
\begin{equation}\small{
\cong K_{2n}^{et} ({\mathcal O}_{F}[1/l]) \cap 
K_{2n}^{et} (F)^{l^k} / K_{2n}^{et} ({\mathcal O}_{F}[1/l])^{l^k} .}
\nonumber\end{equation}
\medskip

\noindent
Hence for every $k \geq 1$ the group 
$D (n, l^k)$ \, (resp. $D^{et} (n, l^k)$ ) is a subquotient of 
$K_{2n} ({\mathcal O}_F)$ ( $K_{2n}^{et} ({\mathcal O}_{F}[1/l])$ resp.). 
Following \cite{Ba1} we will abbreviate our notation at some places as follows: 
\begin{equation}
D(n) := div \, K_{2n} (F) \quad \text{and} \quad 
D^{et}(n) := div \, K_{2n}^{et} (F)_l \, . 
\label{notation D(n)}\end{equation}
\medskip

\noindent
Applying Bockstein sequences (cf. \cite[Diagrams 2.1 and 2.3, p. 289-290]{Ba2})
for all $k \gg 0$ gives: 
\begin{equation}
K_{2n} ({\mathcal O}_F) / l^k =
K_{2n} ({\mathcal O}_F)_l \quad\quad\quad \text{and} \quad\quad\quad  D (n, l^k)  \cong D(n)_l \, . 
\label{1.3}\end{equation}

\noindent
Moreover for all $k \gg 0$ we get by similar argument:
\begin{equation}
K_{2n}^{et} ({\mathcal O}_{F}[1/l]) / l^k =
K_{2n}^{et} ({\mathcal O}_{F}[1/l])_l  \quad\quad \text{and} \quad\quad 
D^{et} (n, l^k)  \cong D^{et}(n) \, .
\label{1.4}\end{equation}
\medskip

\noindent
Let $k(l)$ be the smallest $k$ such that both conditions (\ref{1.3}) and (\ref{1.4}) hold. 
We observe that if $l \not | \,\, |K_{2n} ({\mathcal O}_F)|$ then 
$D (n, l^k)  = D(n)_l = 0$ for all $k \geq 1.$
\bigskip

\noindent
\begin{theorem}\label{Quillen and etale obstr are eq}
If $l > 2$ then $\forall \, k \geq 1$ there is 
the following canonical isomorphism:
\begin{equation}
D (n, l^k) \cong
D^{et} (n, l^k)
\label{D equals D et}\end{equation}
If $l = 2$ then $\forall \, k \geq 2$ 
there is the following canonical isomorphism:
\begin{equation}
D (n, 2^k) \cong
D^{et} (n, 2^k)
\label{D2  equals D2 et}\end{equation} 
If $l \geq  2$ then there is the following isomorphism $D(n)_l \cong  D^{et} (n)$ 
or more explicitly
\begin{equation}
div \, K_{2n}(F)_l \,\, \cong \,\, div \, K_{2n}^{et} (F)_l
\label{div = divetale}\end{equation}
\end{theorem}
\begin{proof}
For every $l$ odd and $k \geq 1$ (resp. for $l = 2$ and $k \geq 2$) 
consider the following commutative diagram.
$$\small{\xymatrix{
\ar[r]^{} & K_{2n+1} (F,\, \Z/l^k)   
\ar@<0.1ex>[d]^{} \ar[r]^{\partial}  & 
 \bigoplus_{v} \, K_{2n} (k_v, \, \Z/l^k) \ar@<0.1ex>[d]^{\cong} \ar[r]^{} & 
 D (n, l^k) \ar[r]^{} \ar@<0.1ex>[d]^{\cong} & 0 \\
\ar[r]^{} & K_{2n+1}^{et} (F,\, \Z/l^k)  \ar[r]^{\partial^{et}}  & 
 \bigoplus_{v \not \, | \, l} \, K_{2n}^{et} (k_v, \, \Z/l^k)  \ar[r]^{} & 
  D^{et} (n, l^k) \ar[r]^{} & 0}}
\label{CoeffLocSeq} $$
The right vertical arrow is an isomorphism because the middle vertical arrow is an isomorphism
[DF, Corollary 8.6]  and the left vertical arrow is an epimorphism 
[DF, Theorem 8.5]. The isomorphism (\ref{div = divetale}) follows from 
(\ref{1.3}), (\ref{1.4}), (\ref{D equals D et}), (\ref{D2  equals D2 et}). 
\end{proof}
\medskip

\begin{corollary} \label{div via sza}
For all $n > 0$ there are the following 
isomorphisms:
\begin{equation}
D(n)_l \, \cong \, D^{et} (n) \, \cong \, D_{n+1} (F) \, = \,  
\Sha^{2}_{S} (F, \Z_l (n+1)) = \Sha^{2} (F, \Z_l (n+1)).
\label{DIVandSZA2}\end{equation}
\end{corollary}
\begin{proof}
This follows by Lemma \ref{SchneiderConj.}, Theorem \ref{functionfieldSchneider3}, Corollary 
\ref{DIVandSZA}, Theorem \ref{Quillen and etale obstr are eq} and by the following 
isomorphism $K_{2n}^{et} (F) \cong H^2 (F, \Z_l (n+1))$ \cite[Prop. 5.1]{DF}.
\end{proof}

\begin{theorem}\label{limitsondivisible} For every $n \geq 1$ there are the following 
isomorphisms:

\noindent
\begin{equation}
\varinjlim_{k}\, D (n, l^k) = 0, 
\label{dirlim}\end{equation}
\begin{equation}
\varprojlim_{k}\, D (n, l^k) \cong D(n)_l.
\label{projlim}\end{equation}
\begin{equation}
\varinjlim_{k}\, D^{et} (n, l^k) = 0, 
\label{dirlimetale}\end{equation}
\begin{equation} 
\varprojlim_{k}\, D^{et} (n, l^k) \cong D^{et}(n)
\label{projlimetale}\end{equation}

\end{theorem}

\begin{proof} The isomorphisms (\ref{dirlim}), (\ref{projlim}), (\ref{dirlimetale}), 
(\ref{projlimetale})
follow by comparing the Bockstein exact sequences in K-theory for ${\mathcal O}_F$
and for $F$ (resp. in {\' e}tale K-theory for ${\mathcal O}_{F}[1/l]$
and for $F$). 
\end{proof}
\medskip

\noindent
\begin{proposition}\label{lim lim1 of direct sum K(kv)1}
For every $n \geq 1:$
\begin{equation}\small{
{\varprojlim_{k}} \, \bigoplus_{v} \, K_{2n} (k_v, \Z / l^k) =
{\varprojlim_{k}} \, \bigoplus_{v} \, K_{2n}^{et} (k_v, \Z / l^k)  = 0 }
\label{lim lim1 of direct sum K(kv)2 }
\end{equation}
and the group
\begin{equation}\small{
{\varprojlim_{k}}^{1} \, \bigoplus_{v} \, K_{2n} (k_v, \Z / l^k) =
{\varprojlim_{k}}^{1} \, \bigoplus_{v} \, K_{2n}^{et} (k_v, \Z / l^k) 
\label{lim lim1 of direct sum K(kv)3 }}
\end{equation}
is torsion free.
\end{proposition}
\begin{proof}
Notice that $K_{2n} (k_v, \Z/l^{k}) \cong K_{2n-1} (k_v) [l^{k}]$ and 
$K_{2n}^{et} (k_v, \Z/l^{k}) \cong K_{2n-1}^{et} (k_v) [l^{k}].$ 
Because $K_{2n} (k_v, \Z/l^{k}) \cong K_{2n}^{et} (k_v, \Z/l^{k})$
by \cite{DF} it is enough to make the proof for K-theory.
Hence (\ref{lim lim1 of direct sum K(kv)2 }) follows because:
\begin{equation}
\small{{\varprojlim_{k}} \, \bigoplus_{v} \, K_{2n} (k_v, \Z / l^k) 
\,\, \subset \,\, {\varprojlim_{k}} \, \prod_{v} \, K_{2n} (k_v, \Z / l^k) = 0.}
\nonumber
\end{equation}
Applying the $lim - lim^{1}$ exact sequence to the 
exact sequence:
\begin{equation}\small{
0 \rightarrow \oplus_{v} \,  K_{2n} (k_v, \Z / l^k) 
\rightarrow \prod_{v} \, K_{2n} (k_v, \Z / l^k) 
\rightarrow \prod_{v} \, K_{2n} (k_v, \Z / l^k) / 
\oplus_{v} \, K_{2n} (k_v, \Z / l^k) 
\rightarrow 0 }
\nonumber
\end{equation} 
gives the natural isomorphism 
\begin{equation}\small{
{\varprojlim_{k}} \prod_{v} \, K_{2n} (k_v, \Z / l^k) / 
\oplus_{v} \, K_{2n} (k_v, \Z / l^k) \,\, \cong \,\,
{\varprojlim_{k}}^{1} \oplus_{v} \, K_{2n} (k_v, \Z / l^k)}
\label{lim lim1 of direct sum K(kv)4}
\end{equation}
The group on the left hand side of (\ref{lim lim1 of direct sum K(kv)4}) 
is clearly torsion free.
\end{proof}

\begin{theorem}\label{lim 1 QL1}
For every $n \geq 1$ there is the following isomorphism:
\begin{equation}\small{
{\varprojlim_{k}}^{1} \, K_{n} (F,\, \Z / l^k) \,\, \stackrel{\cong}{\longrightarrow} \,\, 
{\varprojlim_{k}}^{1} \, K_{n}^{et} (F,\, \Z / l^k).} \label{lim 1 QL2}
\end{equation}
Moreover there is the following equality:
\begin{equation}\small{
{\varprojlim_{k}}^{1} \, K_{2n} (F,\, \Z / l^k) = 0,} \label{lim 1 QL3}
\end{equation}
and the following exact sequence:
\begin{equation}\small{
0 \rightarrow D (n)_l \rightarrow {\varprojlim_{k}}^{1} \, K_{2n+1} (F, \, \Z / l^k)
\rightarrow {\varprojlim_{k}}^{1} \bigoplus_{v} \, K_{2n} (k_v, \Z / l^k)
\rightarrow 0.} \label{lim 1 QL4}
\end{equation}
\end{theorem}

\begin{proof} We are going to give a proof that works for all global fields
satisfying our assumptions set up in section 2.
Consider the following Bockstein exact sequences:
\begin{equation}\small{
0 \rightarrow K_{n} (F) / l^k  \rightarrow  K_{n} (F, \, \Z / l^k) \rightarrow
K_{n-1} (F) [l^k] \rightarrow 0} 
\label{Bockstein1}
\end{equation}
\begin{equation}\small{
0 \rightarrow K_{n}^{et} (F) / l^k  \rightarrow  K_{n}^{et} (F,\, \Z / l^k) \rightarrow
K_{n-1} (F)^{et} [l^k] \rightarrow 0} 
\label{Bockstein2}
\end{equation}
If $n = 2 m$ then $K_{2m -1} (\mathcal{O}_{F, S})_{l} =
K_{2m - 1}(F)_{l}$ and $K_{2m -1}^{et} (\mathcal{O}_{F, S})_l \cong 
K_{2m-1}^{et} (F)_l$ are all finite groups.
Since the natural maps 
$K_{n} (F) / l^{k+1} \rightarrow K_{n} (F) / l^k$ and  
$K_{n}^{et} (F) / l^{k+1} \rightarrow K_{n}^{et} (F) / l^k$  
are surjective for all $n \geq 0$ and all $k \geq 0,$ 
the equality (\ref{lim 1 QL3}) follows by applying
the $lim - lim^{1}$ exact sequence to the Bockstein sequences
(\ref{Bockstein1}) and (\ref{Bockstein2}).
\medskip

\noindent
Consider the natural maps 
$$i \, :\, K_{2n+1} (\mathcal{O}_{F, S},\, \Z/l^k) \rightarrow 
K_{2n+1} (F,\, \Z/l^k),$$ 
$${i^{et} \, :\, K_{2n+1}^{et} (\mathcal{O}_{F, S},\, \Z/l^k) } \rightarrow 
K_{2n+1}^{et} (F,\, \Z/l^k).$$ 
Since the groups $K_{2n+1} (\mathcal{O}_{F, S}, \, \Z/l^k)$
and $K_{2n+1}^{et} (\mathcal{O}_{F, S},\, \Z/l^k)$ are finite, the $lim - lim^{1}$ 
exact sequence shows that 
\begin{equation}
{\varprojlim_{k}}^{1} \, K_{2n+1} (F,\, \Z / l^k)/ 
i( K_{2n+1} (\mathcal{O}_{F S},\, \Z / l^k))
\,\, \cong \,\,{\varprojlim_{k}}^{1} \, K_{2n+1} (F,\, \Z / l^k)
\nonumber
\end{equation} 
\begin{equation}
{\varprojlim_{k}}^{1} \,  K_{2n+1}^{et} (F,\, \Z / l^k) / 
i^{et} ( K_{2n+1}^{et} (\mathcal{O}_{F, S}, \Z / l^k))  \,\, \cong \,\,
{\varprojlim_{k}}^{1} \, K_{2n+1}^{et} (F,\, \Z / l^k).
\nonumber
\end{equation}
Hence taking into account 
Theorem \ref{limitsondivisible} (\ref{projlim}, \ref{projlimetale}), Proposition 
\ref{lim lim1 of direct sum K(kv)1} and applying the $lim - lim^{1}$ exact sequence to the rows of 
following commutative diagram :
$$\small{\xymatrix{
 0  \rightarrow K_{2n+1} (F,\, \Z / l^k) / 
i^{et} ( K_{2n+1}^{et} (\mathcal{O}_{F S},\, \Z / l^k))   
\ar@<0.1ex>[d]^{} \ar[r]^{} & 
 \bigoplus_{v} \, K_{2n} (k_v, \, \Z/l^k) \ar@<0.1ex>[d]^{\cong} \ar[r]^{} & 
 D (n, l^k) \rightarrow \ar@<0.1ex>[d]^{\cong}  0 \\
0 \rightarrow  K_{2n+1} (F,\, \Z / l^k) / 
i^{et} ( K_{2n+1}^{et} (\mathcal{O}_{F S}, \Z / l^k) \ar[r]^{} & 
 \bigoplus_{v \not \, | \, l} \, K_{2n}^{et} (k_v, \, \Z/l^k)  \ar[r]^{} & 
  D^{et} (n, l^k) \rightarrow  0}}
\label{CoeffLocSeq} $$
gives the natural commutative diagram:
$$\small{\xymatrix{
 0  \ar[r]^{} & D (n)_l   
\ar@<0.1ex>[d]^{\cong} \ar[r]^{} & \varprojlim_{k}^{1} \, K_{2n+1} (F,\, \Z / l^k) 
  \ar@<0.1ex>[d]^{} \ar[r]^{} & 
 \varprojlim_{k}^1 \bigoplus_{v} \, K_{2n} (k_v, \, \Z/l^k) 
 \ar[r]^{} \ar@<0.1ex>[d]^{\cong} & 0\\
0 \ar[r]^{} & D^{et} (n) \ar[r]^{}  & \varprojlim_{k}^{1} \, K_{2n+1}^{et} (F,\, \Z / l^k) 
 \ar[r]^{} & \varprojlim_{k}^{1} \bigoplus_{v \not \, | \, l} \, K_{2n}^{et} (k_v, \, \Z/l^k) 
 \ar[r]^{} & 0}}
\label{CoeffLocSeq}$$
Hence the top row of the diagram is the exact sequence (\ref{lim 1 QL4}) and the middle vertical arrow is 
the isomorphism (\ref{lim 1 QL2}) 
\end{proof}
\bigskip

\noindent
Theorem \ref{Quillen and etale obstr are eq} gives the opportunity to
compute the order of the group $D (n)_l.$ 

\begin{example}
Recall that for $n$ odd, $l > 2$ and a 
totally real number field $F$ \cite[Theorem 3 (ii) p. 289]{Ba2} there is the 
following formula: 
\begin{equation}
| D(n)_l |  =
{\big|}\frac{w_{n+1} (F) \zeta_{F} (-n)}{\prod_{v | l} \,  w_n (F_v)} \big{|}_{l}^{-1}.
\label{number of elements in D(n)l 1}\end{equation}
One gets this formula taking $S = S_l,$ applying the equalities 
(\ref{EtKTh OS comp to Div}) and (\ref{div = divetale}), using the theorem of Wiles which
states that: 
$$ | H^2 (\mathcal{O}_{S_l}, \,  \Z_l (n+1))| = |w_{n+1} (F) \zeta_{F} (-n)|_{l}^{-1}.$$
Observe that $|w_{n} (F)|_{l}^{-1} = 1$ for $F$ totally real and $l$ odd. 
\end{example}

\bigskip
\begin{example}
Now let $\text{char}\, F = p > 0$ and let $\F_q$ be the algebraic closure of $\F_p$ in $F.$
Let $X/\F_q$ be a smooth curve corresponding to $F.$ This curve is unique up to 
$\F_q$ isomorphism. Let $Z(X, t)$ denote the Weil zeta function for $X.$ Then put $t = q^{-s}$
and define:
\begin{equation} 
\zeta_{F} (s) := Z(X, q^{-s}). 
\nonumber\end{equation} 

\noindent
The Leray spectral sequence for the natural map
$i\, ;\, \text{spec}\, F \rightarrow X:$ 
$$E_{2}^{i,j} = H^i (X,\, R^{j} i_{\ast} W^{n+1}) \Rightarrow H^{i+j} (F,\, W^{n+1})$$
gives the following exact sequence of the lower terms:
\begin{equation}
0 \, \stackrel{}{\longrightarrow} \,\, H^1 (X, \, W^{n+1}) \stackrel{}{\longrightarrow} 
H^{1} (F,\, W^{n+1}) \stackrel{\partial}{\longrightarrow} \bigoplus_{v} H^{0} (k_v,\, W^{n})
\label{localization equals Moore}\end{equation}
By (\ref{generallocalfieldSchneider113}) and the exact sequences
(\ref{Gen.Sch.F})  and (\ref{MooreEt F}) we obtain for all $n \geq 1$ the following 
natural isomorphism:
\begin{equation} 
D^{et}(n) \cong H^1 (X, \, W^{n+1})
\label{div et equals cohom of X}
\end{equation}
Its is well known c.f. \cite{Ko} p. 202
that $|H^0 (X, \, W^{n+1})| = \big{|} q^{n+1} - 1 \big{|}_{l}^{-1} ,$ \,\, 
$|H^2 (X, \, W^{n+1})| = \big{|} q^{n} - 1 \big{|}_{l}^{-1}$ 
and 
\begin{equation} 
|H^1 (X, \, W^{n+1})| = |(q^{n+1} - 1)(q^{n} - 1) \zeta_{X} (-n)|_{l}^{-1}.
\label{H1 X expresses by zeta of X}
\end{equation}
Since $\F_q$ is the algebraic closure of $\F_p$ in $F$ we have
$W^k (F) = W^k (\F_q)$ for all $k \in \Z.$ In particular for all $k > 0$ we have
$|w_k (F)|_{l}^{-1} = |w_k (\F_q)|_{l}^{-1} = |q^{k} - 1|_{l}^{-1}.$
If $q_v := N v,$ then $|w_k (F_v)|_{l}^{-1} = |w_k (k_v)|_{l}^{-1} = |q_{v}^k - 1|_{l}^{-1}.$
Observe that: 
\begin{equation}
\zeta_{F} (s) = \zeta_{X} (s) \prod_{v \, | \, \infty} \, (1 - N v^{-s})
\nonumber \end{equation} 
Hence by Theorem \ref{Quillen and etale obstr are eq} we get:
\begin{equation} 
| D (n)_l | =  
\big{|} \frac{ w_{n} (F) \, w_{n+1} (F) \, \zeta_{F} (-n)}{\prod_{v \, | \, \infty} w_{n} (F_v)} \big{|}_{l}^{-1}.
\label{D(n) expressed by zeta of F and X}
\end{equation}
\end{example}
\bigskip

\subsection{Application to the homology of GL}

\noindent
Let us return to the general situation where $l \geq 2$ and 
$F$ is a global field of characteristic 
$\text{char}\, F  \not= l$ such that for $l = 2$ it is assumed that 
$\mu_{4} \subset F.$

\begin{theorem}\label{ker on homology OF to F}
For every $n \geq 1,$ $k \geq 1$ and $l > n + 1 $ the kernel of the natural map 
\begin{equation}
H_{2n} (GL(\mathcal{O}_{F}), \, \Z / l^k)  \,\, \rightarrow \,\,
H_{2n} (GL(F), \, \Z / l^k)
\label{ker on homology OF to F 1}
\end{equation}  
contains a subgroup isomorphic to $D (n, l^k).$
\end{theorem}
\begin{proof} Let $A$ be a commutative ring with identity. Comparing the Bockstein exact
sequences for K-theory of $A$ and for the homology of $GL (A)$ and applying the result of
Arlettaz \cite[Cor. 7.19]{Ar2} (cf. \cite{Ar1}) we observe that for all $l > n+1$ and 
$k \geq 1$ the Hurewicz homomorphism 
\begin{equation}
h_{2n} \, :\, K_{2n}(A, \, \Z / l^k)  \,\, \rightarrow \,\,
H_{2n} (GL(A), \, \Z / l^k)
\label{Hurewicz for A}
\end{equation}   
is injective. Hence the claim follows by the following commutative diagram. 
$$\small{\xymatrix{
K_{2n} (\mathcal{O}_{F},\, \Z / l^k)   
\ar@<0.1ex>[d]^{h_{2n}} \ar[r]^{} & K_{2n} (F, \, \Z / l^k) 
  \ar@<0.1ex>[d]^{h_{2n}}\\
H_{2n} (GL(\mathcal{O}_{F}), \, \Z / l^k) \ar[r]^{} & 
H_{2n} (GL(F), \, \Z / l^k)}}
\label{CoeffLocSeq}$$ 
\end{proof}

\begin{corollary}\label{cor1 ker on homology OF to F}
Let $n  \geq 1$ and $l > n+1.$ Assume that 
$K_{2n}^{et} (\mathcal{O}_{F}[1/l]) \cong D^{et} (n)_l$ and 
$l \,\, || \,\, |D (n)_l|.$
Then kernel of the natural map 
\begin{equation}
H_{2n} (GL(\mathcal{O}_{F}), \, \Z / l) \,\, \rightarrow \,\,
H_{2n} (GL(F), \, \Z / l)
\label{ker on homology OF to F 2}
\end{equation}  
contains a subgroup isomorphic to $D (n)_l .$
\end{corollary} 
\begin{proof} By Theorem \ref{Quillen and etale obstr are eq} 
we have $D(n)_l \cong D^{et} (n)_l$ and $D(n, l) \cong D^{et} (n, l).$  
Moreover by 
(\ref{1.11}) and the assumptions we have the 
following isomorphism $D^{et}(n)_l \cong D^{et} (n, l),$ hence 
$D(n)_l \cong D (n, l).$
\end{proof}

\begin{corollary}\label{cor2 ker on homology OF to F}
Let $F = \Q$ and let $n  \geq 1$ be odd. 
Assume that $l > n + 1 $ is such that  
$l \,\, || \,\, |w_{n+1} (\Q) \zeta_{\Q} (-n)|_{l}^{-1}.$
Then the kernel of the natural map 
\begin{equation}
H_{2n} (GL(\Z), \, \Z / l) \,\, \rightarrow \,\,
H_{2n} (GL(\Q), \, \Z / l)
\label{ker on homology OF to F 3}
\end{equation}  
contains a subgroup isomorphic to $\Z/l.$
\end{corollary} 
\begin{proof} Observe that
$K_{2n}^{et} (\Z[1/l]) \cong D^{et} (n)_l$ by Theorem \ref{MooreEt} equality 
(\ref{EtKTh OS comp to Div}) (see also \cite{Ba2}).
Moreover $|D(n)_l| \, = \, |D^{et} (n)_l| \, = \, |w_{n+1} (\Q) \zeta_{\Q} (-n)|_{l}^{-1}$
by \cite[Theorem 3 p. 289]{Ba2}.
Hence the claim follows by Corollary \ref{cor1 ker on homology OF to F}.
\end{proof}

\begin{example} Let $F = \Q,$ $n = 11$ and $l = 691.$ 
Observe that
$w_{12} (\Q) \zeta_{\Q} (-11) = 2 \times 691$ 
cf. \cite[p. 343]{Ba1}.
Then the kernel
of the natural map: 
\begin{equation}
H_{22} (GL(\Z), \, \Z / 691) \,\, \rightarrow \,\,
H_{22} (GL(\Q), \, \Z / 691)
\nonumber
\end{equation}
contains a subgroup isomorphic to $\Z / 691.$  
\end{example}

\begin{example} Let $F = \Q,$ $n = 15$ and $l = 3617.$ Observe that
$w_{16} (\Q) \zeta_{\Q} (-15) = 2 \times 3617$ 
cf. \cite[Example, p. 358]{Ba1}. Hence the kernel
of the natural map: 
\begin{equation}
H_{30} (GL(\Z), \, \Z / 3617) \,\, \rightarrow \,\,
H_{30} (GL(\Q), \, \Z / 3617)
\nonumber
\end{equation}  
contains a subgroup isomorphic to $\Z / 3617.$
\end{example}
\bigskip

Let $E/\F_q$ be an elliptic curve given by the Weierstrass equation $y^2 = x^3 + Ax + B$ with
$A, B \in \F_q.$ Let $F := \F_q (E).$ Since $[F, \, \F_q (x)] = 2$ the finite field $\F_q$
is algebraically closed in $F.$ There is only one point at $\infty$ and $F_{\infty} =
\F_q ((x)).$ Moreover $|w_{n} (F_{\infty})|_{l}^{-1} = |w_{n} (\F_q)|_{l}^{-1}
= |q^{n} - 1|_{l}^{-1}.$ 
It was proven by Weil (see [Sil]) that for $a := 1 + q - |E (\F_q)|$ there is the following formula:
\begin{equation} 
Z(E, q^{-s}) = \frac{1 - a q^{-s} + q^{1 - 2s}}{(1 - q^{-s})(1 - q^{1-s})} 
\nonumber\end{equation} 
For an elliptic curve over $\F_q$ the formula 
(\ref{EtKTh OS comp to Div}) of the Theorem \ref{MooreEt}  
yields the following equality
$K_{2n}^{et} (\mathcal{O}_{F}) \cong D^{et} (n)_l.$  
\bigskip

Now let $q = p$ a prime number and let $E/\F_p$ be supersingular. Since 
$|a| \leq 2 \sqrt{p}$ by Hasse theorem and $a \equiv 0 \mod p$ then for $p > 3$ we   
have $a = 0.$ Hence taking $s = -n$ and using (\ref{D(n) expressed by zeta of F and X})
we get: 
\begin{equation} 
| D (n)_l | =  
\frac{\big{|} 1 + p^{1 + 2n}\big{|}_{l}^{-1}}{\big{|} 1 - p^n \big{|}_{l}^{-1}} .
\label{D(n) for supersingular}
\end{equation}

\begin{corollary}\label{cor3 ker on homology OF to F}
Let $E/\F_p$ be a supersingular elliptic curve and let 
$F = \F_p (E).$ Let $n  \geq 1$ be odd and $l$ be a prime number 
such that $p \equiv -1 \mod l,$  $l > n+1$ and 
$l  \not | \,\, \frac{(p+1)(2n+1)}{l}.$ Then the kernel of the natural map: 
\begin{equation}
H_{2n} (GL(\mathcal{O}_{F}), \, \Z / l) \,\, \rightarrow \,\,
H_{2n} (GL(F), \, \Z / l)
\label{ker on homology OF to F 4}
\end{equation}  
contains a subgroup isomorphic to $\Z/l.$
\end{corollary} 
\begin{proof} It is clear that for $n$ odd and $l$ such that $p \equiv -1 \mod l$ and 
 $l  \not | \,\, \frac{(p+1)(2n+1)}{l}$ the formula 
(\ref{D(n) for supersingular}) yields $l \,\, || \,\, |D (n)_l|.$
Hence the claim follows by Corollary \ref{cor1 ker on homology OF to F}.
\end{proof}

\begin{example} The elliptic curve $y^2 = x^3 +1$ over $\F_p,$ 
for $p \geq 5,$ is supersingular iff $p \equiv 2 \mod 3$ \cite[pp. 143-144]{Si}. 
In particular for $p = 29,$ $l = 5,$ $n$ odd
and $n \not\equiv 2 \mod 5$ we notice by 
(\ref{D(n) for supersingular})
that $5 \,\, || \,\, |D (n)_5|.$ Hence by Corollary 
\ref{cor3 ker on homology OF to F} the kernel of the following map:
\begin{equation}
H_{6} (GL(\mathcal{O}_{\F_{29} (E)}), \, \Z / 5) \,\, \rightarrow \,\,
H_{6} (GL(\F_{29}(E)), \, \Z / 5)
\label{ker on homology OF to F 5}
\end{equation}  
contains a subgroup isomorphic to $\Z / 5.$ 
Similarly for $p = 41,$ $l = 7$ and 
any $n$ odd such that $n \not\equiv 3 \mod 7$ we notice by 
(\ref{D(n) for supersingular}) that $7 \,\, || \,\, |D (n)_7|.$ Hence by 
Corollary \ref{cor3 ker on homology OF to F} the kernel of the following map: 
\begin{equation}
H_{10} (GL(\mathcal{O}_{\F_{41}(E)}), \, \Z / 7) \,\, \rightarrow \,\,
H_{10} (GL(\F_{41}(E)), \, \Z / 7)
\label{ker on homology OF to F 6}
\end{equation}  
contains a subgroup isomorphic to $\Z / 7.$ 
\end{example}

\begin{example} The elliptic curve $y^2 = x^3 + x$ over $\F_p,$ 
for $p \geq 3,$ is supersingular iff $p \equiv 3 \mod 4,$ \cite[pp. 143-144]{Si}. 
Again taking $p = 19,$ $l = 5,$ $n$ odd 
and $n \not\equiv 2 \mod 5$ we notice by (\ref{D(n) for supersingular})
that $5 \,\, || \,\, |D (n)_5|.$ So by 
Corollary \ref{cor3 ker on homology OF to F}  
the kernel of the following map:
\begin{equation}
H_{6} (GL(\mathcal{O}_{\F_{19}(E)}), \, \Z / 5) \,\, \rightarrow \,\,
H_{6} (GL(\F_{19}(E)), \, \Z / 5)
\label{ker on homology OF to F 6}
\end{equation}  
contains a subgroup isomorphic to $\Z / 5.$  
\end{example}


\section{The wild kernels and divisible elements}

\subsection{Wild kernels and the Moore exact sequence}
The following theorem is basically known however the results are scattered over a number of papers. 
Namely, surjectivity of the map (\ref{surjective 1}) is due to \cite{DF} and 
surjectivity of (\ref{surjective 2}) for number 
fields was proven in \cite{Ba2}. The splitting of the map (\ref{surjective 1}) in the number field case 
was settled in \cite{Ba2} and the canonical splitting of the map 
(\ref{surjective 1}) in global field case was settled in 
\cite{K}. The splitting of the map (\ref{surjective 2}) for the even K-groups of number field was 
proven in \cite{Ca}. For the record we make a very short proof of Theorem \ref{H2 of F} pointing 
out key ingredients.

\begin{theorem}\label{H2 of F}
For every $n \geq 1$ and
every finite set $S \supset S_{l}$  
the following natural maps are split surjective:
\begin{equation}
K_n ({\mathcal O}_{F, S})_l \rightarrow K_{n}^{et} ({\mathcal O}_{F, S})_l .
\label{surjective 1}
\end{equation} 
\begin{equation}
K_n (F)_l \rightarrow K_{n}^{et} (F)_l.  
\label{surjective 2}
\end{equation}
\end{theorem}
\begin{proof}
If $X$ denotes ${\mathcal O}_{F, S}$ or $F$ then by 
\cite{DF} Theorem 8.5
the left vertical arrow in the following commutative diagram is surjective.
\begin{equation}
\xymatrix{
K_{n+1} (X,\, \Z/l^k)   
\ar@<0.1ex>[d]^{} \ar[r]^{} & K_{n} (X) [l^k]
 \ar[r]^{} \ar@<0.1ex>[d]^{} & 0 \\
K_{n+1}^{et} (X,\, \Z/l^k)  \ar[r]^{} & K_{n}^{et} (X) [l^k] 
 \ar[r]^{} & 0}
\label{BocksteinDiag1}
\end{equation}
Hence the right vertical arrow is surjective so 
$K_{n} (X)_l \rightarrow K_{n}^{et} (X)_l$ is surjective 
cf. \cite[Theorem 1]{Ba2} . 
The surjectivity of the map (\ref{surjective 2}) follows also 
by surjectivity
of the maps (\ref{surjective 1}) for all finite $S$ upon taking the 
direct limit over $S,$ (cf. \cite{Ba2} the proof of Theorem 1).
Since the groups $K_{n} ({\mathcal O}_{F, S})_l,$ 
$K_{n}^{et} ({\mathcal O}_{F, S})_l,$ 
are finite for all $n > 0$, and the groups
$K_{n} (F)_l$ and $K_{n}^{et} (F)_l$
are finite for all $n$ odd then 
the splitting for the map (\ref{surjective 1}) for all $n > 0$ (resp. 
the splitting of the map (\ref{surjective 2}) for all $n$ odd)
follows from the investigation of the right vertical arrow of the 
diagram (\ref{BocksteinDiag1}), cf. the proof of \cite[Proposition 2]{Ba2} .
For the splitting of the map (\ref{surjective 2}) with $n$ even we use the 
method of Luca Caputo \cite{Ca}. Namely from the diagram 
(\ref{BocksteinDiag1}) we find out that the kernel of the map (\ref{surjective 2})
is a pure subgroup of $K_{2n} (F)_l$ and from the diagram of the proof of
Lemma \ref{SchneiderConj.} we get that this kernel is finite. Hence by 
\cite[Theorem 7]{Ka} the map (\ref{surjective 2}) is split surjective.
Actually the splitting of both maps (\ref{surjective 1}) and (\ref{surjective 2}) 
for all $n \geq 1$ follows by this method by use of \cite[Theorem 7]{Ka}. 
   
\end{proof}

The Wild kernels $K_{n}^{w} ({\mathcal O}_{F})_l$ and $WK_n (F)$ are defined as the kernels of the 
natural localization maps to make the following sequences exact:

\begin{equation}
0 \rightarrow K_{n}^{w} ({\mathcal O}_{F})_l \rightarrow
K_{n} (F)_l
\rightarrow \prod_{v} \, K_{n}^{et} (F_v)_l
\label{WildKer1}\end{equation}
\begin{equation}
0 \rightarrow WK_{n} (F) \rightarrow
K_{n} (F)
\rightarrow \prod_{v} \, K_{n} (F_v)
\label{WildKer2}\end{equation}

where the product is over all places of $F.$ 
\begin{remark}
By our general assumption in this paper,
we will consider the $2$-part of $WK_{n} (F)$ only for such $F,$ 
for which $\mu_{4} \subset F,$
although the definition of $WK_{n} (F)$ is for any global field.
\label{2 part of WKnF}\end{remark}

\begin{lemma} \label{arch infty not important for wild}
If $F$ is a number field then $WK_{n} (F)_l$ maps to zero via 
\begin{equation}
K_{n} (F) \rightarrow \prod_{v \in S_{\infty}} \, K_{n} (F_v)
\label{F to infty1}
\end{equation}
for every $l \geq 2.$
\end{lemma}
\begin{proof}
Choose an imbedding ${\overline F} \subset \C.$ Take any
$v^{\prime} \notin S_{\infty}.$

\noindent
If $l = 2$ then by assumption $\Q (i)\subset F$ so $F_v = \C$ for 
all $v \in S_{\infty}.$ Then
$F \subset F_{v^{\prime}}^h \subset {\overline F} \subset F_v = \C$ for every $v \in S_{\infty}.$ The map (\ref{F to infty1}) factors as follows: 

\begin{equation}
K_{n} (F) \rightarrow K_{n} (F_{v^{\prime}}^h) \rightarrow \prod_{v \in S_{\infty}} \, K_{n} (F_v). 
\label{F to infty2} \end{equation}
and since $K_{n} (F_{v^{\prime}}^h) \rightarrow K_{n} (F_{v^{\prime}})$ is a monomorphism
\cite{BZ2} then $WK_{n} (F)_{2}$ maps to zero via the left arrow of (\ref{F to infty2}). 

\noindent
If $l > 2$ then for any $v \in S_{\infty}$ such that $F_v = \R$ consider the imbedding 
$F_v \subset \C.$ The group $WK_{n} (F)_{l}$ maps to zero via the left map 
of composition of maps:
\begin{equation}
K_{n} (F) \rightarrow K_{n} (F_{v^{\prime}}^h) \rightarrow \prod_{v \in S_{\infty}} \, K_{n} (\C). 
\label{F to infty3} \end{equation}
The map $\prod_{v \in S_{\infty}} \, K_{n} (F_v) \rightarrow 
\prod_{v \in S_{\infty}} \, K_{n} (\C)$ is an imbedding on the $l$-torsion part
hence $WK_{n} (F)_{l}$ maps to zero via the left map of the composition 
\begin{equation}
K_{n} (F) \rightarrow  \prod_{v \in S_{\infty}} \, K_{n} (F_v) 
\rightarrow \prod_{v \in S_{\infty}} \, K_{n} (\C). 
\label{F to infty4} \end{equation}
\end{proof}

\begin{lemma} \label{arch infty not important for wild etale}
If $F$ is a number field then $K_{n}^{w} ({\mathcal O}_{F})_l$ maps to zero via 
\begin{equation}
K_{n} (F)_l \rightarrow \prod_{v \in S_{\infty}} \, K_{n}^{et} (F_v)_l
\label{F to infty1}
\end{equation}
for every $l \geq 2.$
\end{lemma}
\begin{proof} For every $v^{\prime} \notin S_{\infty}$ there is a natural
isomorphism $G_{F_{v^{\prime}}^h} \cong
G_{F_{v^{\prime}}}.$ Consequently there is the natural isomorphism
$H^{i} (F_{v^{\prime}}^h, \, \Z_{l} (j)) \cong H^{i} (F_{v^{\prime}}, \, \Z_{l} (j)).$
So the map $\text{spec} \, F_{v^{\prime}} 
\rightarrow \text{spec} \, F_{v^{\prime}}^h$ gives the 
isomorphism of Dwyer-Friedlander spectral sequences \cite[Prop. 5.1]{DF} which yields 
the isomorphism: 
\begin{equation}
K_{n}^{et} (F_{v^{\prime}}^h) \cong K_{n}^{et} (F_{v^{\prime}}).
\label{etale K of Fvh and Fv are the same}\end{equation}
Now the proof is very similar to the proof of Lemma 
\ref{arch infty not important for wild} applying 
(\ref{etale K of Fvh and Fv are the same}) in place of \cite{BZ1}.  
\end{proof}
\begin{remark}
Lemmas \ref{arch infty not important for wild} and 
\ref{arch infty not important for wild etale} show that the wild
kernels defined in this paper agree with the wild kernels defined in
\cite{Ba2} and \cite{BGKZ} in the number field case .
\label{earlier wild agree with the present wild}\end{remark}

\medskip

Observe that $WK_n(F) \subset K_n ({\mathcal O}_{F})$
for any global field $F$ and $n \geq 0$ cf. \cite{BGKZ}.
In the number field case it was proved in \cite{BGKZ}
that $WK_n(F)$ is torsion. In the function field case 
$K_n ({\mathcal O}_{F})$ is torsion for all $n > 1$ and it is clear that 
$WK_0 (F) = WK_1 (F) = 0.$
Hence for any global field $F$ and any $n \geq 0$ we get:
\begin{equation}
WK_n(F) \subset K_n ({\mathcal O}_{F})_{tor}
\label{WildKerContainedinTorofTameKer}
\end{equation}
Hence in particular if $F$ is a number field then 
the group $K_{2n}^{w} ({\mathcal O}_{F})_l$ has already been 
defined in \cite{Ba2} and  
the group $WK_n (F)$ has been defined in \cite{BGKZ}. 
\medskip

Consider the following commutative diagram.
$$\xymatrix{
0 \ar[r]^{}  & WK_{n} (F)_l   
\ar@<0.1ex>[d]^{} \ar[r]^{}  &  K_{n} (F)_l        
 \ar@<0.1ex>[d]^{=} \ar[r]^{} & \prod_{v} \, K_{n} (F_v)_l
\ar[r]^{} \ar@<0.1ex>[d]^{} &  \\
0 \ar[r]^{} & K_{n}^{w} ({\mathcal O}_{F})_l   \ar[r]^{}  & K_{n} (F)_l 
 \ar[r]^{} & \prod_{v} \, K_{n}^{et} (F_v)_l
 \ar[r]^{} & }
\label{WildKerComparison} $$
From this diagram we notice that for any $n > 0$ and any $l \geq 2$  we have:
 
\begin{equation}
WK_n (F)_l \subset K_{n}^{w} ({\mathcal O}_{F})_l \subset K_{n} ({\mathcal O}_{F})_l.
\end{equation}  
Hence by \cite{Ta2}, Proposition 2.3 p. 261 we observe that
\begin{equation}
K_{2n}^{et} (F_v)_l  \cong H^1 (F_v, \Q_l / \Z_l (n+1)) / Div \cong W^n (F_v),
\end{equation}
\begin{equation}
K_{2n+1}^{et} (F_v)_l  \cong H^0 (F_v, \Q_l / \Z_l (n+1)) / Div \cong W^{n+1} (F_v),
\end{equation}
Hence the group $K_{n}^{et} (F_v)_l$ is finite for any $v,$ any $n \geq 1$ and $l \geq 2.$  
This shows that 
\begin{equation}
div \, K_{n} (F)_l \,  \subset  \, K_{n}^{w} ({\mathcal O}_{F})_l \, 
\subset \, K_{n} ({\mathcal O}_{F})_l.
\end{equation}  
Applying Quillen localization sequences for rings ${\mathcal O}_{F,S}$
and ${\mathcal O}_{v}$ it is also important to notice, that for any finite 
set $S \supset S_{l}$
there are the following exact sequences:
\begin{equation}
0 \rightarrow WK_{n} (F)_l \rightarrow K_{n} ({\mathcal O}_{F,S})_l 
\rightarrow \prod_{v \in S} \, K_{n} ({\mathcal O}_v)_l
\label{WildKerViaIntegral1}
\end{equation}

\begin{equation}
0 \rightarrow  K_{n}^{w} ({\mathcal O}_{F})_l   \rightarrow K_{n} ({\mathcal O}_{F,S})_l 
\rightarrow \prod_{v \in S} \, K_{n}^{et} ({\mathcal O}_v)_l
\label{WildKerViaIntegral2}
\end{equation}

The following theorem gives the analog of the classical 
Moore exact sequence for higher K-groups of global fields.

\begin{theorem} \label{Moore exact sequence}
For every $n \geq 1$ and every finite set 
$S \supset S_{\infty, l}$ there are the following exact sequences:
\begin{equation}
\small{0 \rightarrow K_{2n}^{w} ({\mathcal O}_{F})_l \rightarrow
K_{2n} ({\mathcal O}_{F,S})_l
\rightarrow \bigoplus_{v \in S} W^n (F_v)
\rightarrow W^{n} (F) \rightarrow
0.}
\label{Moore exact sequence1}\end{equation}
\begin{equation}
\small{0 \rightarrow K_{2n}^{w} ({\mathcal O}_{F})_l \rightarrow
K_{2n} (F)_l
\rightarrow \bigoplus_{v} W^n (F_v)
\rightarrow W^{n} (F) \rightarrow
0.}
\label{Moore exact sequence2}\end{equation}
In particular:
\begin{equation}
\small{\frac{|K_{2n} (\mathcal{O}_{F, S})_l|}{|K_{2n}^{w} (\mathcal{O}_{F})_l|} 
= {\bigl|}\frac{\prod_{v \in S} w_{n} (F_v)}{w_{n} (F)}{\bigr|}_{l}^{-1}.}
\label{Kth OS comp to Wild}\end{equation}
\end{theorem}
\begin{proof}
It results from Theorems \ref{MooreEt} and \ref{H2 of F}. The equality
(\ref{Kth OS comp to Wild}) follows from (\ref{Moore exact sequence1})
since all terms in this exact sequence are finite.
\end{proof} 

\begin{lemma} \label{div as ker}
For every $n \geq 1$ and every $l \geq 2$ there is the following exact sequence
\begin{equation}
0 \rightarrow div \, K_{n}^{et}(F)_l    \rightarrow   K_{n}^{et} (F)_l 
 \rightarrow  \prod_{v} \, K_{n}^{et} (F_v)_l 
\label{div as ker sequence}\end{equation}
\end{lemma}
\begin{proof}
Put $n = 2i - j$ for $j = 1, 2.$ 
Hence by (\ref{DF SpecSeqResult1}), by \cite[Theorem 3.2]{Ja}  and by 
\cite[Proposition 2.3 p. 261]{Ta2},  the exact sequence has the following form:
\begin{equation}
\small{0 \rightarrow div \, K_{2i - j}(F)_l    \rightarrow   
H^{j-1} (G_F, \Q_l/ \Z_l (i)) / Div  \rightarrow  \prod_{v} \,  
H^{j-1} (G_{F_v}, \Q_l/\Z_l (i)) / Div.} 
\label{div as ker sequence 1}
\end{equation}
Let $j = 1.$ Then the map $H^{0} (G_F, \Q_l/ \Z_l (i)) / Div 
\rightarrow  H^{0} (G_{F_v}, \Q_l/\Z_l (i)) / Div$ is trivially injective for each $v$
and  $div \, K_{2i - 1}(F) = 0$ so (\ref{div as ker sequence 1}) is exact in this case.
For $j = 2$ the exactness of (\ref{div as ker sequence 1}) is the result of Theorem
\ref{functionfieldSchneider3}.
\end{proof}

Consider the following commutative diagrams:
$$\xymatrix{
0 \ar[r]^{}  & K_{n}^{w} ({\mathcal O}_{F})_l   
\ar@<0.1ex>[d]^{} \ar[r]^{}  &  K_{n} (F)_l        
 \ar@<0.1ex>[d]^{} \ar[r]^{} & \prod_{v} \, K_{n}^{et} (F_v)_l
 \ar@<0.1ex>[d]^{=} &  \\
0 \ar[r]^{} & div \, K_{n}^{et} (F)_l    \ar[r]^{}  & K_{n}^{et} (F)_l 
 \ar[r]^{} & \prod_{v} \, K_{n}^{et} (F_v)_l &}
\label{WildKer onto div}$$

$$\xymatrix{
0 \ar[r]^{}  & K_{n}^{w} ({\mathcal O}_{F})_l   
\ar@<0.1ex>[d]^{} \ar[r]^{}  &  K_{n} ({\mathcal O}_{F,S})_l        
 \ar@<0.1ex>[d]^{} \ar[r]^{} & \prod_{v} \, K_{n}^{et} (F_v)_l
 \ar@<0.1ex>[d]^{=} &  \\
0 \ar[r]^{} & div \, K_{n}^{et} (F)_l    \ar[r]^{}  & K_{n}^{et} ({\mathcal O}_{F,S})_l 
 \ar[r]^{} & \prod_{v} \, K_{n}^{et} (F_v)_l &}
\label{WildKer onto div OS}$$
The left vertical arrows in both diagrams are identical.

\begin{theorem} \label{wild maps on div via a split map 1}
The left vertical arrows in the diagrams above are split surjective.
The middle vertical arrows induce canonical isomorphisms for all $n > 1:$
\begin{equation}
K_{n}(\mathcal{O}_{F,S})_l / K_{n}^{w}(\mathcal{O}_{F})_l 
\,\, \stackrel{\cong}{\longrightarrow} \,\,
K_{n}^{et}(\mathcal{O}_{F,S})_l / div K_{n}^{et} (F)_l  
\label{QL holds beyond wild kernel formula1}
\end{equation}
\begin{equation}
K_{n}(F)_l / K_{n}^{w}(\mathcal{O}_{F})_l 
\,\, \stackrel{\cong}{\longrightarrow} \,\,
K_{n}^{et}(F)_l / div K_{n}^{et} (F)_l  
\label{QL holds beyond wild kernel formula2}
\end{equation}
\end{theorem}
\begin{proof}
It follows since the middle vertical arrows in the diagrams above 
are split surjective by Theorem \ref{H2 of F}. 
\end{proof} 
\medskip

\subsection{Divisible elements, wild kernels and Quillen-Lichtenbaum conjecture}
We keep working with global fields as stated in the introduction. 
\begin{conjecture} (Quillen-Lichtenbaum) Let $F$ be a global 
field and let $l \geq 2.$ Assume that $\mu_{4} \subset F$ if $l = 2.$ 
Then for all $n > 1$ and $l \not= \, \text{char} \, F$ 
the natural map:
\begin{equation}
K_n (\mathcal{O}_{F}) \otimes \Z_l 
\rightarrow K_{n}^{et} (\mathcal{O}_{F}[\frac{1}{l}])
\end{equation}
is an isomorphism.
\end{conjecture}
\medskip

It has been known for many years that the Quillen-Lichtenbaum conjecture can be reformulated in several ways.  
Theorem \ref{QLequivalent} below presents some of the ways to refermulate this conjecture. 

\begin{theorem}\label{QLequivalent}
The following conditions are equivalent:
\begin{enumerate} 
\item{} $K_{n} (\mathcal{O}_{F}) \otimes \Z_l \stackrel{\cong}{\longrightarrow} 
K_{n}^{et} (\mathcal{O}_{F}[\frac{1}{l}])$ for all $n > 1,$
\item{} $K_{n} (F)_l \stackrel{\cong}{\longrightarrow} 
K_{n}^{et} (F)_l$ for all $n > 1,$
\item{}
$K_{n} (\mathcal{O}_{F}, \Z / l^k) \stackrel{\cong}{\longrightarrow}
K_{n}^{et} (\mathcal{O}_{F} [\frac{1}{l}], \Z / l^k)$
for all $k > 0$ and $n > 1.$
\item{}
$K_{n} (F, \Z / l^k) \stackrel{\cong}{\longrightarrow}
K_{n}^{et} (F, \Z / l^k)$
for all $k > 0$ and all $n > 1.$
\item{}
${\varprojlim_{k}} \, K_{n} (F, \Z / l^k) \stackrel{\cong}{\longrightarrow}
{\varprojlim_{k}} \, K_{n}^{et} (F, \Z / l^k)$
for all $n > 1.$
\item{}
$K_{n}^{cts} (F, \Z_l) \stackrel{\cong}{\longrightarrow}
K_{n}^{et} (F)$
for all $n > 1.$
\item{} $K_{n}^{w}(\mathcal{O}_{F})_l = div K_n (F)_l$ for all $n > 1.$  
\end{enumerate} 
\end{theorem}
\begin{proof} The equivalence of conditions (1), (2), (3) and  (4)
follows by finite generation of
K-groups of $\mathcal{O}_{F}$ and by comparison of Bockstein and localization 
sequences for Quillen and {\' e}tale K-theory. Clearly (4) implies (5).  
Consider the following commutative diagram cf. \cite{BZ1}: 
$$\xymatrix{
0 \ar[r]^{}  & {\varprojlim_{k}}^{1} \, K_{n+1} (F, \Z / l^k)  
\ar@<0.1ex>[d]^{\cong} \ar[r]^{}  &  K_{n}^{cts} (F, \Z_l)        
 \ar@<0.1ex>[d]^{} \ar[r]^{} & {\varprojlim_{k}} \, K_{n} (F, \Z / l^k)
\ar[r]^{} \ar@<0.1ex>[d]^{} & 0 \\
0 \ar[r]^{} & {\varprojlim_{k}}^{1} \, K_{n+1}^{et} (F, \Z / l^k)    \ar[r]^{}  & K_{n}^{et} (F) 
 \ar[r]^{} & {\varprojlim_{k}} \, K_{n}^{et} (F, \Z / l^k)
 \ar[r]^{} & 0}
\label{lim1 cont to etale} $$
where $K_{n}^{cts} (F, \Z_l)$ is the continuous K-theory defined in \cite{BZ1}.
Hence (5) and Theorem \ref{lim 1 QL1} implies that the middle vertical arrow in this diagram is an 
isomorphism. By \cite{BZ1} Theorem 1 this implies (2). Hence we proved that (5)
implies (4). This diagram also shows that (5) and (6) are equivalent. By the diagram 
following the proof of Lemma \ref{div as ker} conditions (2) and (7) are equivalent.  
\end{proof}

\noindent
Base on the proof of Theorem \ref{QLequivalent} we easily prove the following theorem.
\begin{theorem}\label{QLequivalentMorePrecise}
For every $n > 1$ the following conditions are equivalent:
\begin{enumerate} 
\item{} $K_{n} (\mathcal{O}_{F}) \otimes \Z_l \stackrel{\cong}{\longrightarrow} 
K_{n}^{et} (\mathcal{O}_{F}[\frac{1}{l}]),$
\item{} $K_{n} (F)_l \stackrel{\cong}{\longrightarrow} 
K_{n}^{et} (F)_l,$
\item{}
${\varprojlim_{k}} \, K_{n} (F, \Z / l^k) \stackrel{\cong}{\longrightarrow}
{\varprojlim_{k}} \, K_{n}^{et} (F, \Z / l^k),$
\item{}
$K_{n}^{cts} (F, \Z_l) \stackrel{\cong}{\longrightarrow}
K_{n}^{et} (F),$
\item{} $K_{n}^{w}(\mathcal{O}_{F})_l = div K_n (F)_l.$  
\end{enumerate} 
\end{theorem}
\begin{proof}
Exercise for the reader.
\end{proof}
\medskip

It is well known that the important results on motivic cohomology due to S. Bloch, E. Friedlander, 
M. Levine, S. Lichtenbaum, F. Morel, M. Rost, A. Suslin, V. Voevodsky, C. Weibel and others 
(see e.g. \cite{BL} cf. \cite[Appendix B]{RW}, \cite{L}, \cite{MV}, \cite{R}, \cite{V1}, \cite{V2},
\cite{VSF}, \cite{We4})
prove the Quillen-Lichtenbaum conjecture.
Let us state this in the following theorem. 
 
\begin{theorem} The Quillen-Lichtenbaum conjecture holds true for $F$.
In particular the equivalent conditions in Theorems \ref{QLequivalent} and \ref{QLequivalentMorePrecise} hold true for $F.$  \label{Quillen-Lichtenbaum conj. equivalent conditions}
\end{theorem}

\noindent
\begin{remark}
The proof of Theorem \ref{Quillen-Lichtenbaum conj. equivalent conditions} applies, as some of the 
key ingredients, the spectral 
sequence connecting the motivic cohomology and K-theory 
\cite{BL} cf. \cite[Appendix B]{RW}: 
\begin{equation}
E_{2}^{p,q} = H^{p-q}_{\mathcal{M}} (F, \,\Z/l^k (-q) ) \Rightarrow K_{-p-q} (F, \, \Z/l^k)
\label{motivic spectral sequence}
\end{equation} 
and results of Voyevodsky 
(\cite[Theorem 7.9]{V1} for $l = 2$ and  \cite[Theorem 6.16]{V2} for $l > 2$)
giving: 
\begin{enumerate} 
\item{} 
$H^{j}_{\mathcal{M}} (K, \Z/l^k (i)) \cong H^{j}_{et} (K, \Z/l^k (i))$
for all $j \leq i,$ 
\item{} $H^{j}_{\mathcal{M}} (F, \Z/l^k (i)) = 0$ if $j > i,$
\end{enumerate}
for any field $K$ of characteristic not equal to $l$.
Then one can connect these results with the Dwyer and Friedlander spectral sequence 
\cite[Proposition 5.2]{DF} to get the following isomorphism:
\begin{equation}
K_{n} (F, \Z/l^k) \cong K_{n}^{et} (F, \Z/l^k),
\label{QL}\end{equation} for all $l \geq 2,$ which is equivalent with the 
Quillen-Lichtenbaum conjecture (see Theorem \ref{QLequivalent}). 
\label{proof of QL conjecture}\end{remark}

\noindent
\begin{remark}
In 1992 M. Levine \cite{L} observed that Bloch-Kato conjecture 
for fields implies the Quillen-Lichtenbaum conjecture. The Bloch-Kato conjecture is proved in 
\cite[Corollary 7.5]{V1} (for $l = 2$) and \cite[Theorem 6.16]{V2} (for all $l$).
It is shown in \cite{GL} that Bloch-Kato conjecture lead to the computation of 
the motivic cohomology in terms of \'etale cohomology, a property that 
also establishes the Quillen-Lichtenbaum conjecture as pointed out in Remark 
\ref{proof of QL conjecture}.
\end{remark}
\medskip

The Proposition \ref{K-th to etale K-th for local rings} and its 
Corollary \ref{K-th to etale K-th for local fields} are well known 
and follow from the Gabber rigidity, results of Suslin \cite{Su1}
and results of Dwyer and Friedlander \cite{DF}. We include here short proofs.
\begin{proposition} \label{K-th to etale K-th for local rings}
Let $l$ be prime to $\text{char} \,\, k_v.$ There are natural isomorphisms:

\begin{equation}
K_{n}(\mathcal{O}_v, \, \Z/l^k)  \stackrel{\cong}{\longrightarrow}   K_{n}^{et} (\mathcal{O}_v, \, \Z/l^k)  
\label{K(Ov coeff lk) }
\end{equation}
\begin{equation}
K_{n}(\mathcal{O}_v) [l^k]  \stackrel{\cong}{\longrightarrow}   K_{n}^{et} (\mathcal{O}_v) [l^k]  
\label{K(Ov) lk trancation}
\end{equation}
\begin{equation}
K_{n}(\mathcal{O}_v) /l^k  \stackrel{\cong}{\longrightarrow}   K_{n}^{et} (\mathcal{O}_v) /l^k  
\label{K(Ov) mod lk}
\end{equation}
\end{proposition}

\begin{proof}
Consider the following commutative diagram.
$$\xymatrix{
K_{n} (\mathcal{O}_v, \Z / l^k)  
\ar@<0.1ex>[d]^{\cong} \ar[r]^{\cong}  &  K_{n}^{et} (\mathcal{O}_v, \Z / l^k)        
 \ar@<0.1ex>[d]^{\cong}\\
K_{n} (k_v, \Z / l^k)  \ar[r]^{\cong}  &  K_{n}^{et} (k_v, \Z / l^k)}
\label{Quillen to etale local fields1} $$
The left vertical arrow is an isomorphism by \cite[Corollaries 2.5 and 3.9]{Su1} . 
The bottom vertical arrow is an isomorphism by \cite[Corollary 8.6]{DF} .
Hence the top horizontal arrow is a monomorphism. Moreover, the top horizontal arrow 
is an epimorphism by \cite{DF} Theorem 8.5 and by comparison of K-theory and 
\' etale K-theory localization sequences with coefficients for the local ring 
$\mathcal{O}_v.$ 
The maps (\ref{K(Ov) lk trancation}), (\ref{K(Ov) mod lk}) are isomorphisms by comparison 
of Bockstein K-theory and \' etale K-theory sequences for $\mathcal{O}_v$
with corresponding Bockstein sequences for $k_v$ cf. \cite{BGKZ} Section 2.  
\end{proof} 

\begin{corollary} \label{K-th to etale K-th for local fields}
Let $l$ be prime to $\text{char} \,\, k_v.$ There are natural isomorphisms:

\begin{equation}
K_{n}(F_v, \, \Z/l^k) \stackrel{\cong}{\longrightarrow}   K_{n}^{et} (F_v, \, \Z/l^k)  
\label{K(Fv coeff lk)}
\end{equation}
\begin{equation}
K_{n}(F_v) [l^k] \stackrel{\cong}{\longrightarrow}   K_{n}^{et} (F_v) [l^k]  
\label{K(Fv) lk trancation}
\end{equation}
\begin{equation}
K_{n}(F_v)/l^k \stackrel{\cong}{\longrightarrow}   K_{n}^{et} (F_v)/l^k  
\label{K(Fv) mod lk}
\end{equation}
\end{corollary}

\begin{proof}
The isomorphism (\ref{K(Fv coeff lk)}) follows by Proposition \ref{K-th to etale K-th for local rings}
and by comparison of K-theory and \' etale K-theory localization sequences with coefficients.
Consider the following commutative diagram with exact rows:
$$\xymatrix{
0 \ar[r]^{}  &  \, K_{n} (\mathcal{O}_v)_l  
\ar@<0.1ex>[d]^{\cong} \ar[r]^{}  &  K_{n} (F_v)_l         
 \ar@<0.1ex>[d]^{} \ar[r]^{} & \, K_{n-1} (k_v)_l 
\ar[r]^{} \ar@<0.1ex>[d]^{\cong} & 0 \\
0 \ar[r]^{}  &  \, K_{n}^{et} (\mathcal{O}_v)_l  \ar[r]^{}  &  K_{n}^{et} (F_v)_l         
 \ar[r]^{} & \, K_{n-1}^{et} (k_v)_l 
\ar[r]^{} & 0 }
\label{Quillen to etale local case 1}$$
The bottom exact sequence is an appropriate {\' e}tale cohomology exact sequence
written in terms of \' etale K-theory. It follows by Proposition
\ref{K-th to etale K-th for local rings} that the map (\ref{K(Fv) lk trancation})
is an isomorphism, hence by Bockstein sequence argument the map (\ref{K(Fv) mod lk})
is also an isomorphism.
\end{proof}

\begin{remark}
If $p = \text{char} \,\, k_v$ then it was proven in \cite{HM} that:
\begin{equation}
K_{n}(F_v, \, \Z/p^k) \stackrel{\cong}{\longrightarrow}   K_{n}^{et} (F_v, \, \Z/p^k).  
\label{K(Fv coeff pk)}
\end{equation}
By Bockstein sequence argument the map
\begin{equation}
K_{n}(F_v) [p^k] \stackrel{\cong}{\longrightarrow}   K_{n}^{et} (F_v) [p^k]  
\label{K(Fv) pk trancation}
\end{equation} 
is an epimorphism and the map
\begin{equation}
K_{n}(F_v)/p^k \stackrel{\cong}{\longrightarrow}   K_{n}^{et} (F_v)/p^k  
\label{K(Fv) mod pk}
\end{equation}
is a monomorphism.
\end{remark}

Consider the following commutative diagram.
$$\xymatrix{
0 \ar[r]^{}  & WK_{n} (F)_l   
\ar@<0.1ex>[d]^{} \ar[r]^{}  &  K_{n} (F)_l        
 \ar@<0.1ex>[d]^{} \ar[r]^{} & \prod_{v} \, K_{n} (F_v)_l
 \ar@<0.1ex>[d]^{} &  \\
0 \ar[r]^{} & div \, K_{n} (F)_l    \ar[r]^{}  & K_{n}^{et} (F)_l 
 \ar[r]^{} & \prod_{v} \, K_{n}^{et} (F_v)_l &}
\label{WildKerCompDiv}$$

\begin{theorem}\label{wild maps on div via a split map 2}
Assume that for every $v \in S_l:$
\begin{equation} 
K_{n} (F_v)_l 
\stackrel{\cong}{\longrightarrow} K_{n}^{et} (F_v)_l .
\label{local QL assumption for the bad primes}\end{equation} 
 Then for all $n \geq 1$ the left vertical arrow in
the diagram above is split surjective. 
Moreover the following conditions are equivalent for all $n \geq 1$:
\begin{enumerate} 
\item{} $K_{n} (F)_l \stackrel{\cong}{\longrightarrow} 
K_{n}^{et} (F)_l,$ 
\item{} $WK_n(F)_l = div K_n (F)_l$
\end{enumerate}  
\end{theorem}
\begin{proof}
By theorem \ref{H2 of F} the middle vertical arrow is split surjective.
The right vertical arrow is an isomorphism by Corollary \ref{K-th to etale K-th for local fields} 
and our assumption. This shows that the left vertical arrow is split surjective. 
Hence the left vertical
arrow is an isomorphism if and only if the middle vertical arrow is an isomorphism.
\end{proof}
\medskip

\noindent 
Under assumption (\ref{local QL assumption for the bad primes}) the Theorem 
\ref{Quillen-Lichtenbaum conj. equivalent conditions} shows that
the equivalent conditions in Theorem 
\ref{wild maps on div via a split map 2} hold true.
Summing up results in this section concerning wild kernels and divisible elements, 
we state the following theorem.

\begin{theorem}\label{div equals wild kernel due to QL true} Let $l \geq 2.$ 
For every $n > 1$ and we have the following equality: 
\begin{equation} 
K_{n}^{w}(\mathcal{O}_{F})_l = div K_n (F)_l. 
\nonumber
\end{equation} 
Assume that $K_{n} (F_v)_l 
\stackrel{\cong}{\longrightarrow} K_{n}^{et} (F_v)_l$ for every $v \in S_l$ and every
$n > 1.$ Then for every $n \geq 0:$  
\begin{equation} 
WK_n(F)_l = div\, K_n (F)_l.
\nonumber\end{equation}  
\end{theorem}
\medskip

\noindent
\begin{remark}
It is easy to observe that $WK_n(F) = div\, K_n (F) = 0$
for $0 \leq n \leq 1.$
\end{remark}


\section{Splitting obstructions to Quillen boundary map}

Observe that the \cite[Diagram 2.5]{Ba2} and the corresponding diagram for {\' e}tale 
K-theory and also [Ba2], Diagram 3.2 extend naturally to the global field case
and $l \geq 2.$ Hence by analogues arguments as the ones in 
loc. cit. we get for every $k \geq 1,$ every $n \geq 1$ 
the following commutative diagram with exact rows:   

$$\small{\xymatrix{
\dots \ar[r]^{}  & K_{2n} (F)[l^k]   
\ar@<0.1ex>[d]^{} \ar[r]^{}  & 
\oplus_{v} \, K_{2n-1} (k_v)[l^k] \ar@<0.1ex>[d]^{\cong} \ar[r]^{} & 
 D (n, l^k) \ar[r]^{} \ar@<0.1ex>[d]^{\cong} & 0 \\
 \dots \ar[r]^{} & K_{2n}^{et} (F)[l^k]   \ar[r]^{}  & 
 \oplus_{v \not \, | \, l} \, K_{2n-1}^{et} (k_v)[l^k]  \ar[r]^{} & 
  D^{et} (n, l^k) \ar[r]^{} & 0}}
\label{CoeffLocSeq} $$
Actually the rows of this diagram have the following form: 
\begin{equation}\small{
0 \rightarrow K_{2n} ({\mathcal O}_F)[l^k] \rightarrow 
K_{2n} (F) [l^k] 
\rightarrow \oplus_{v} \, K_{2n-1} (k_v)[l^k] \rightarrow 
D (n, l^k) \rightarrow 0.}
\label{1.2}\end{equation} 

\begin{equation}\small{
0 \rightarrow K_{2n}^{et} ({\mathcal O}_F)[l^k] \rightarrow 
K_{2n}^{et} (F) [l^k] 
\rightarrow \oplus_{v} \, K_{2n-1}^{et} (k_v)[l^k] \rightarrow 
D^{et} (n, l^k) \rightarrow 0.}
\label{1.22}\end{equation} 
Taking direct limit in (\ref{1.2}) gives the $l$-part of the Quillen 
localization sequence 
\begin{equation}\small{
0 \rightarrow K_{2n} ({\mathcal O}_F)_l \rightarrow 
K_{2n} (F)_l
\stackrel{\partial}{\longrightarrow} 
\oplus_{v} \, K_{2n-1} (k_v)_l \rightarrow 0.}
\label{QuillenLocSeq}\end{equation} 
which also implies the property (\ref{dirlim}).
\medskip

\noindent
Recall the definition of the numbers $k (l)$ in section 4. Define 
\begin{equation}
N_{0} := \prod_{l \,\,\, \big{|} \,\,\, |K_{2n} ({\mathcal O}_F)|} \,\, l^{k (l)}.
\label{Nzero}
\end{equation} 
The exact sequence (\ref{1.2}) for every $l$ shows that for every 
positive integer $N$ such that $N_{0} \, | \, N$ 
we have the following exact sequence:
\begin{equation}\small{
0 \rightarrow K_{2n} ({\mathcal O}_F) \rightarrow 
K_{2n} (F) [N] 
\rightarrow \bigoplus_{v} \, K_{2n-1} (k_v)[N] \rightarrow 
D (n) \rightarrow 0.}
\label{analogofclassgroup}\end{equation} 
The exact sequence (\ref{analogofclassgroup}) shows that the group 
$D (n)$ is an analog for higher K-groups of the class group $Cl ({\mathcal O}_F).$
Recall that the class group appears in the exact sequence:
\begin{equation}\small{
0 \rightarrow K_{1} ({\mathcal O}_F) \rightarrow 
K_{1} (F) 
\rightarrow \bigoplus_{v} \, K_{0} (k_v) \rightarrow 
 Cl ({\mathcal O}_F) \rightarrow 0.}
\label{classgroup}\end{equation}
\medskip

\noindent
\begin{remark}
To determine whether a map of two $l$-torsion abelian groups 
is split surjective I considered in \cite[ p. 293 and p. 296]{Ba2} 
obstructions to the splitting via $l^k$ truncations of this map.
Working throughout \cite{Ba2} with $k \gg 0$ 
I did not consider on p. 293 loc. cit. the cokernels of the $l^k$
truncation of $\partial$ for $k < k(l).$ By
(\ref{1.2}) the cokernel of the $l^k$ trancation of 
$\partial$ is $D(n, l^k)$ and in particular for 
$k \geq k(l)$ we have $D(n)_l = D(n, l^k).$
As a result in \cite[Corollary 1, p. 293]{Ba2}  
I have an incomplete statement. The Proposition \ref{splitting1} below 
completes the statement of \cite[Corollary 1, p. 293]{Ba2}.
The proof of the Proposition \ref{splitting1} below is the same as 
the the proof of \cite[Corollary 1, p. 293]{Ba2} by 
considering $l^k$ truncations for all $k > 0$ not just for 
$k \geq k(l).$  The gap in the statement of [Ba2] Corollary 1 p. 293  
has been noticed by Luca Caputo in \cite{Ca}.   
\end{remark}

\begin{proposition}\label{splitting1}
The following conditions are equivalent:
\begin{enumerate} 
\item{} $D (n, l^k) = 0$ \,\, for every \,\, $0 < k \leq k(l),$    
\item{} $K_{2n} (F)_l \, \cong \, K_{2n} ({\mathcal O}_F)_l \, \oplus \,  
\bigoplus_{v} \, K_{2n-1} (k_v)_l$
\end{enumerate} 
\end{proposition}

\begin{proof} (2) implies 
$K_{2n} (F)[l^k] \, \cong \, K_{2n} ({\mathcal O}_F)[l^k] \, \oplus \,  
\bigoplus_{v} \, K_{2n-1} (k_v) [l^k]$ for every $k > 0$ hence 
$D(n, l^k) =0$ for every $k > 0$ by (\ref{1.2}).
\medskip

\noindent
Now assume (1). By definition of $k (l)$ we note that $D(n, l^k) = 0$ for every 
$0 < k \leq k(l)$ if and only if $D(n, l^k) = 0$ for every $k > 0.$ 
Hence by (\ref{1.2}) there is an exact sequence for every $k > 0:$
\begin{equation}
0 \rightarrow K_{2n} ({\mathcal O}_F)[l^k] \rightarrow 
K_{2n} (F) [l^k] 
\rightarrow \bigoplus_{v} \, K_{2n-1} (k_v)[l^k] \rightarrow 0.
\label{exact truncated}\end{equation}
The groups $K_{2n-1} (k_v)$ are finite cyclic. Hence for every $v$
we can choose $k \geq 0$ that $K_{2n-1} (k_v)_l = K_{2n-1} (k_v)[l^k].$
Hence the exact sequence (\ref{exact truncated}) 
allows us to construct a homomorphism
$\Lambda_v \, : K_{2n-1} (k_{v})_l \rightarrow 
K_{2n} (F)_l$ such that for every element $\xi_v \in K_{2n-1} (k_{v})_l$
we get $\partial (\Lambda_v (\xi_v)) = (\dots, 1, \xi_v, 1, \dots) \in
\bigoplus_{v} \, K_{2n-1} (k_{v})_l.$
Hence the map 
$$\Lambda := \prod_v \, \Lambda_v,$$ 
$$\Lambda\, :\, \bigoplus_{v} \, K_{2n-1} (k_{v})_l \, \rightarrow \, K_{2n} (F)_l$$
clearly splits $\partial$ in the Quillen localization sequence (\ref{QuillenLocSeq}).
\end{proof}

\begin{proposition}\label{splitting1etale}
The following conditions are equivalent:
\begin{enumerate} 
\item{} $D^{et} (n, l^k) = 0$ \,\, for every \,\, $0 < k \leq k(l),$    
\item{} $K_{2n}^{et} (F)_l \, \cong \, K_{2n}^{et} ({\mathcal O}_{F}[1/l])_l \, \oplus \,  
\bigoplus_{v} \, K_{2n-1}^{et} (k_v)_l$
\end{enumerate} 
\end{proposition}
\begin{proof} The proof is precisely the same as the proof of 
Proposition \ref{splitting1} with use of the exact sequence (\ref{1.22}). 
\end{proof}

\begin{theorem}\label{splitting2}
The following conditions are equivalent:
\begin{enumerate} 
\item{} $D (n, l^k) = 0$ \,\, for every \,\, $0 < k \leq k(l),$  
\item{} $D^{et} (n, l^k) = 0$ \,\, for every \,\, $0 < k \leq k(l),$  
\item{} $K_{2n} (F)_l \, \cong \, K_{2n} ({\mathcal O}_F)_l \, \oplus \,  
\bigoplus_{v} \, K_{2n-1} (k_v)_l,$
\item{} $K_{2n}^{et} (F)_l \, \cong \, K_{2n}^{et} ({\mathcal O}_{F}[1/l])_l \, \oplus \,  
\bigoplus_{v} \, K_{2n-1}^{et} (k_v)_l.$
\end{enumerate} 
\end{theorem}
\begin{proof} It follows by Theorem \ref{Quillen and etale obstr are eq},
Propositions \ref{splitting1}, \ref{splitting1etale} and the definition of $k (l).$
\end{proof}
\medskip

\noindent
Observe that for any totally real number field 
$F$ any odd $n > 0$ and any odd prime number $l$ 
we have 
$$ | K_{2n}^{et} ({\mathcal O}_F [1/l]) | = 
|{w_{n+1} (F) \zeta_{F} (-n)} |_{l}^{-1}.$$
The following corollary is a correction of \cite{Ba2} Proposition 1 p. 293.

\begin{corollary}\label{splitting3}
Let $n$ be an odd positive integer and let $l$ be an odd prime number. 
Let $F$ be a totally real number field such that 
$\prod_{v | l} \,  w_n (F_v) = 1.$
The following conditions are equivalent:
\begin{enumerate} 
\item{} The following exact sequence splits
\begin{equation}
0 \rightarrow K_{2n} ({\mathcal O}_F)_l \rightarrow 
K_{2n} (F)_l
\stackrel{\partial}{\longrightarrow} 
\oplus_{v} \, K_{2n-1} (k_v)_l \rightarrow 0.
\nonumber\end{equation}
\item{} 
$$|{w_{n+1} (F) \zeta_{F} (-n)} |_{l}^{-1} = 1$$  
\end{enumerate} 
\end{corollary}

\begin{proof}
In our case $ |D(n)_l|  =
|w_{n+1} (F) \zeta_{F} (-n) |_{l}^{-1},$ (see (\ref{number of elements in D(n)l 1})).
Moreover, by Theorem \ref{Quillen and etale obstr are eq}, for every $k > 0$ we have 
$D(n, l^k) \cong D^{et}(n, l^k)$ and 
$D^{et}(n, l^k)$ is a subquotient of $K_{2n}^{et} ({\mathcal O}_F [1/l]).$ In addition, 
as we observed before, $D(n)_l \cong D(n, l^k)$ for $k \gg 0.$ 
Hence in the assumptions of the corollary $D(n, l^k) = 0$ for all
$k > 0$ iff $|w_{n+1} (F) \zeta_{F} (-n) |_{l}^{-1} = 1.$ 
Hence the corollary follows by Corollary \ref{splitting1}.
\end{proof}

\begin{corollary}\label{splitting4}
Let $n$ be an odd positive integer 
and let $l$ be an odd prime number.
The following conditions are equivalent:
\begin{enumerate} 
\item{} The following exact sequence splits
\begin{equation}
0 \rightarrow K_{2n} (\Z) \rightarrow 
K_{2n} (\Q)
\stackrel{\partial}{\longrightarrow} 
\oplus_{v} \, K_{2n-1} (k_v) \rightarrow 0.
\nonumber\end{equation}
\item{} 
$$|{w_{n+1} (\Q) \zeta_{\Q} (-n)} |_{l}^{-1} = 1$$  
\end{enumerate} 
\end{corollary}
\begin{proof} It follows from Corollary \ref{splitting3}
since $|w_n (\Q_{l})|_{l}^{-1} = 1.$
\end{proof}
\medskip

\noindent
\begin{remark}
Take $F = \F_p (x).$ Then ${\mathcal O}_F = \F_p [x].$ 
By the homotopy invariance \cite{Q1} Corollary p. 122 we have 
$K_n (\F_p [x]) = K_n (\F_p).$
Hence the boundary map in the localization sequence gives
the following isomorphism:
$$K_{2n}(\F_p (x)) \stackrel{\cong}{\longrightarrow}
\oplus_{v} \, K_{2n-1} (k_v)$$
In particular $D(n) = div K_{2n}\, (\F_p (x)) = 0.$  
\end{remark}

Let $l^{k_0}$ be the exponent of the group $K_{2n} ({\mathcal O}_F)_l.$ 
\medskip

\noindent
\begin{lemma}\label{nerrowlocalization1} 
For every $k \geq 1$ and every $k^{\prime} \geq k + k_0:$ 

\begin{enumerate} 
\item{} the natural map \,\,
$D (n, l^k) \rightarrow D (n, l^{k^{\prime}})$ \,\, is trivial,
\item{}
$\bigoplus_{v} \, K_{2n-1} (k_v)[l^k] \subset \partial (K_{2n} (F) [l^{k^{\prime}}]).$
\end{enumerate} 
\end{lemma}

\begin{proof}
Statement (1) follows from the commutative diagram with exact rows:
$$\xymatrix{
0 \quad \ar[r]^{}  \quad & 
\quad D (n, l^k)   
\ar@<0.1ex>[d]^{} \ar[r]^{}  \quad & 
\quad K_{2n} ({\mathcal O}_F) / l^k \ar@<0.1ex>[d]^{l^{k^{\prime} - k}} \ar[r]^{}  \quad & 
\quad K_{2n} (F) / l^{k} \ar@<0.1ex>[d]^{l^{k^{\prime} - k}} \\
0 \quad \ar[r]^{}  \quad & 
\quad D (n, l^{k^{\prime}}) \ar[r]^{}  \quad & 
\quad K_{2n} ({\mathcal O}_F) / l^{k^{\prime}}  \ar[r]^{}  \quad & 
\quad K_{2n} (F) / l^{k^{\prime}}}
\label{}$$
since the middle vertical map is trivial by definition of $k_0.$
\medskip

\noindent
Statement (2) follows from (1) and the following commutative diagram with exact rows:
$$\xymatrix{
K_{2n} (F)[l^k] \quad \ar[r]^{\partial} \ar@<0.1ex>[d]^{} \quad & 
\quad \bigoplus_{v} \, K_{2n-1} (k_v)[l^k]    
\ar@<0.1ex>[d]^{} \ar[r]^{}  \quad & 
\quad D (n, l^k) \ar@<0.1ex>[d]^{0} \ar[r]^{}  \quad & 
\quad 0 \\
K_{2n} (F)[l^{k^{\prime}}] \quad \ar[r]^{\partial}  \quad & 
\quad \bigoplus_{v} \, K_{2n-1} (k_v)[l^{k^{\prime}}] \ar[r]^{}  
\quad & \quad D (n, l^{k^{\prime}}) \ar[r]^{}  \quad & 
\quad 0 }
\label{}$$
since the left and the middle vertical arrows are natural imbeddings.
\end{proof}
\bigskip

\noindent
For any $k \geq k(l)$ let us define 
$${\bigoplus_{v}}^{(1)} \, K_{2n-1} (k_v)_l := 
\bigoplus_{l^{k} \, | \, q_{v}^n - 1} \, K_{2n-1} (k_v)_l$$
$${\bigoplus_{v}}^{(2)} \, K_{2n-1} (k_v)_l := 
\bigoplus_{l^{k} \, {\not |} \, q_{v}^n - 1} \, K_{2n-1} (k_v)_l$$

\begin{theorem}\label{splitting5}
Let $F$ be a global field, $n \geq 1$ and $l$ be any prime number.
The following conditions are equivalent:
\begin{enumerate} 
\item{} $D (n)_l = 0,$ 
\item{} the following surjective map splits
$$\partial_{1} \, : \, K_{2n} (F)_l
\rightarrow {\bigoplus_{v}}^{(1)} \, K_{2n-1} (k_v)_l .$$
\end{enumerate} 
\end{theorem}
\begin{proof}
For each $k^{\prime} \geq k$ consider the following exact sequence 
\begin{equation} 
K_{2n} (F) [l^{k^{\prime}}]
\stackrel{\partial}{\longrightarrow} 
{\bigoplus_{v}}^{(1)} \, K_{2n-1} (k_v) [l^{k^{\prime}}] \oplus 
{\bigoplus_{v}}^{(2)} \, K_{2n-1} (k_v) [l^{k^{\prime}}]
\rightarrow D(n)_l \rightarrow 0.
\label{QuillenLocSeqApplic1}\end{equation} 
where $\partial = \partial_1 \oplus \partial_2.$
Assume that $D (n)_l = 0.$ Hence for each $k^{\prime} \geq k$ the following map is 
surjective:
\begin{equation} 
K_{2n} (F) [l^{k^{\prime}}]
\stackrel{\partial_1}{\longrightarrow} 
{\bigoplus_{v}}^{(1)} \, K_{2n-1} (k_v) [l^{k^{\prime}}].
\label{QuillenLocSeqApplic2}\end{equation}  
So for each $v$ such that $l^{k} \, | \, q_{v}^n - 1$ we take $k^{\prime} \geq k$ 
such that $l^{k^{\prime}} ||  q_{v}^n - 1$ and we notice that there is a homomorphism 
$$\Lambda_v \, : \, K_{2n-1} (k_v)_l \, \rightarrow \, K_{2n} (F)_l$$ 
such that 
$$\partial_1 \circ \Lambda_v (\xi_v) \,\, = \,\,  
(\dots, 1, \, \xi_v, \, 1, \dots),$$
for any $\xi_v \in K_{2n-1} (k_v)_l.$ It is clear that
$$\Lambda_1 \, :\, {\bigoplus_v}^{(1)} \, 
K_{2n-1} (k_v)_l \, \rightarrow \, K_{2n} (F)_l$$
$$\Lambda_1 \, := \, {\prod_v}^{(1)} \,\, \Lambda_v$$ 
splits $\partial_1.$
\medskip

\noindent
Assume now that $\partial_1$ is split surjective.
Consider the exact sequence (\ref{QuillenLocSeqApplic1})
for $k^{\prime} = k + k_0.$ For such $k^{\prime}$ by Lemma 
\ref{nerrowlocalization1} we have  
$$\partial_{2} (K_{2n} (F)[l^{k^{\prime}}]) \subset {\bigoplus_{v}}^{(2)} \, 
K_{2n-1} (k_v) [l^{k^{\prime}}] = $$
\begin{equation} 
= {\bigoplus_{v}}^{(2)} \, K_{2n-1} (k_v) [l^{k}] 
\subset {\bigoplus_{v}} \, K_{2n-1} (k_v) [l^{k}]  
\subset \partial (K_{2n} (F)[l^{k^{\prime}}]).
\label{imagesofboundrymaps}\end{equation} 
On the other hand 
$\partial_{1}$ is split surjective hence
${\bigoplus_{v}}^{(1)} \, K_{2n-1} (k_v) [l^{k^{\prime}}]
= \partial_{1} (K_{2n} (F)[l^{k^{\prime}}]).$ Since 
$\partial = \partial_{1} \oplus \partial_{2}$ then by 
(\ref{imagesofboundrymaps}) we see that 
$\partial_{1} (K_{2n} (F)[l^{k^{\prime}}]) \subset 
\partial (K_{2n} (F)[l^{k^{\prime}}]).$ Hence again by (\ref{imagesofboundrymaps}) 
the map $\partial$ in the exact sequence (\ref{QuillenLocSeqApplic1})
is surjective. Hence $D (n)_l = 0.$
\end{proof}

\begin{corollary}\label{splitting6}
Let $F$ be a totally real number field. Let $n$ be an odd positive integer 
and let $l$ be an odd prime number.
The following conditions are equivalent:
\begin{enumerate} 
\item{} The following surjective map splits
$$\partial_{1} \, : \, K_{2n} (F)_l
\rightarrow {\bigoplus_{v}}^{(1)} \, K_{2n-1} (k_v)_l $$
\item{} 
$$\big{|}\frac{w_{n+1} (F) \zeta_{F} (-n)}{\prod_{v | l} \,  w_n (F_v)} \big{|}_{l}^{-1} = 1$$  
\end{enumerate} 
\end{corollary}
\begin{proof}
By \cite{Ba2} Theorem 3 p. $| D(n)_l |  =
|\frac{w_{n+1} (F) \zeta_{F} (-n)}{\prod_{v | l} \,  w_n (F_v)} |_{l}^{-1}.$
Hence the corollary follows by Theorem \ref{splitting5}.
\end{proof}

\begin{corollary}\label{splitting7}
Let $F$ be a global field of $\text{char} \, F > 0$. Let $n > 1$ be an integer 
and let $l \not= \, \text{char} \, F$.
The following conditions are equivalent:
\begin{enumerate} 
\item{} The following surjective map splits
$$\partial_{1} \, : \, K_{2n} (F)_l
\rightarrow {\bigoplus_{v}}^{(1)} \, K_{2n-1} (k_v)_l $$
\item{} 
$$ 
\big{|} \frac{ w_{n} (F) \, w_{n+1} (F) \, \zeta_{F} (-n)}{\prod_{v \, | \, \infty} w_{n} 
(F_v)} \big{|}_{l}^{-1} = 1.
$$
\end{enumerate} 
\end{corollary}
\begin{proof}
Due to (\ref{D(n) expressed by zeta of F and X}) the corollary follows by Theorem 
\ref{splitting5}.
\end{proof}

\medskip

\noindent
{\it Acknowledgments}:\quad
I would like to thank the Max Planck Institute in Bonn
for hospitality and financial support during visit in April and May 2010
and SFB in Muenster for hospitality and financial support during visit in 
September 2010. I would also like to thank the University
of California, San Diego for hospitality and financial
support during visit December 2010 -- June 2011.
\bigskip


\bibliographystyle{amsplain}

\begin{thebibliography}{W}

\bibitem[Ar1]{Ar1}  D.  Arlettaz,\, 
\textit{The Hurewicz homomorphism in algebraic K-theory}, 
Jour. of Pure and Appl. Algebra \textbf{71} (1991), 1-12

\bibitem[Ar2]{Ar2}  D.  Arlettaz,\, 
\textit{Algebraic K-theory of rings from a topological viewpoint}, 
Publicacions Mathem\' atiques \textbf{44}, Servei de Publicacions, 
Univ. Aut. Barcelona (2000), 3-84



\bibitem[A]{A} M. Artin,\, 
\textit{Grothendieck Topologies}, 
        \textbf{} (1962) Springer-Verlag Berlin,  Heidelberg, New York

\bibitem[Ba1]{Ba1} G. Banaszak,\,  
\textit{Algebraic K-theory of number fields, rings of integers 
and the Stickelberger ideal}, Annals of Math. 
       \textbf{135} (1992), 325-360
       
\bibitem[Ba2]{Ba2} G. Banaszak,\,  
\textit{Generalization of the Moore exact sequence and the wild kernel 
for higher K-groups}, Compositio Math. 
       \textbf{86} (1993), 281-305

       
\bibitem[BZ1]{BZ1} G. Banaszak, P. Zelewski\,  
\textit{Continuous K-theory}, K-theory  
       \textbf{9, No 4} (1995), 379-393      

\bibitem[BZ2]{BZ2} G. Banaszak, P. Zelewski\,  
\textit{On the map between K-groups of henselization and completion of some local rings},  
J. of Pure and Appl. Algebra \textbf{120} (1997), 161-165

\bibitem[BG1]{BG1} G. Banaszak, W. Gajda \, 
\textit{Euler Systems for Higher K - theory of number fields} 
Journal of Number Theory,
\textbf{58, No 2} (1996),  213-256 

\bibitem[BG2]{BG2} G. Banaszak, W. Gajda \, 
\textit{On the arithmetic of cyclotomic fields and the K-theory 
of ${\bf Q}$}, Proceedings of the Conference on Algebraic K-theory, 
Contemporary Math. AMS, \textbf{199} (1996), 7-18, G. Banaszak, W. Gajda 
and P. Krason editors.

\bibitem[BGKZ]{BGKZ} G. Banaszak, W. Gajda, P. Kraso\' n, P. Zelewski
\textit{A note on Quillen-Lichtenbaum Conjecture and arithmetic
of square rings}, K-theory \textbf{16} (1999), 229-243

\bibitem[BP]{BP} G. Banaszak, C. Popescu\, 
\textit{The Stickelberger splitting map and Euler systems in the $K$--theory of number fields}, 
To appear in Journal of Number Theory 
\textbf{} 

\bibitem[B]{B} H. Bas
\textit{$K_2$ des corps globaux}, S{\' e}minare Bourbaki 23e ann{\' e}e \textbf{394} (1970/1971), 
233-255

\bibitem[BL]{BL} S. Bloch, S. Lichtenbaum\,
\textit{A spectral sequence for motivic cohomology}, Preprint \textbf{} (1994)

\bibitem[Bo]{Bo} A. Borel,\,
\textit{Stable real cohomology of arithmetic groups},
Ann. Sci. \' Ecole Nor. Sup.
       \textbf{7 (4)} (1974), 235-272

\bibitem[Ca]{Ca} L. Caputo,\,   
\textit{Splitting in the K-theory localization sequence of number fields}, 
J. Pure Appl. Algebra \textbf{215 n. 4}, (2010), 485-495



       

\bibitem[CS]{CS} J. Coates, W. Sinnott,\,
\textit{An analogue of Stickelberger's theorem for the higher K-groups }, 
Invent. Math. \textbf{24} (1974), 149-161   

       
\bibitem[DF]{DF} W. Dwyer, E. Friedlander,\, 
\textit{Algebraic and \' etale K-theory}, Trans. Amer. Math. Soc.
       \textbf{292} (1985), 247-280

\bibitem[GL]{GL} T. Geisser, M. Levine, \,
\textit{The Bloch-Kato conjecture and a theorem of Suslin-Voyevodsky}, 
J. Reine Angew. Math., \textbf{530} no.2 (2001), 55-103




\bibitem[G]{G} D. Grayson,\, 
\textit{Finite generation of groups $K$ of a curve over a finite field}
(after D. Quillen) 
Lecture Notes in Mathematics
        \textbf{551} (1982), 217-240, Springer-Verlag

\bibitem[Ha]{Ha} G. Harder,\,  
\textit{Die Kohomologie S-arithmetischer Gruppen {\" u}ber 
FunktionenK{\" o}rper}, Invent. Math. 
       \textbf{42} (1977), 135-175
       
\bibitem[HM]{HM} L. Hesselholt, I. Madsen\,  
\textit{On the K-theory of local fields}, Annals of Math. 
\textbf{158} (2003), 1-113

\bibitem[Hu]{Hu} K. Hutchinson\,  
\textit{\' Etale wild kernels of exceptional number fields},  
Journal of Number Theory, \textbf{120, No 2} (2006), 47-71 

\bibitem[Ja]{Ja} U. Jannsen,\,  
\textit{Continuous \' etale cohomology}, Math. Ann. 
       \textbf{280} (1988), 207-245       

\bibitem[K]{K} B. Kahn,\,  
\textit{On the Lichtenbaum-Quillen conjecture}, Algebraic K-theory and algebraic 
topology (Lake Louise, AB, 1991), 147-166, NATO Adv. Sci. Inst. Ser. C Math. phys. Sci.,
\textbf{407} (1993) Kluwer Acad. publ., Dordrecht 

\bibitem[Ka]{Ka} I. Kaplansky,\, 
\textit{Infinite abelian groups}, University of Michigan Press, Ann Arbor vol. \textbf{}
(1954)

\bibitem[Ko]{Ko} M. Kolster,\,  
\textit{K-theory and arithmetic}, ICTP Lecture Notes series,  
\textbf{15} (2003) Contemporary Developments in Algebraic K-Theory
M. Karoubi, A.O. Kuku, C. Pedrini eds.

\bibitem[L]{L} M. Levine,\,
\textit{Relative Milnor K-theory}, 
K-theory \textbf{6} no.2 (1992), 113-175

\bibitem[Me]{Me} A. Merkurev,\,  
\textit{On the torsion in $K_2$ of local fields}, 
Annals of Math. \textbf{118} (1983), 375-381
       
\bibitem[MS]{MS} Merkurev, A. Suslin,\,  
\textit{$K$-cohomology of Severi-Brauer varieties and the norm residue homomorphism. }, 
Izv. Akad. Nauk SSSR Ser. Mat.,
\textbf{46 no. 5}, (1982), 1011-1046 (Russian) 

\bibitem[M1]{M1} J.S. Milne,\, 
\textit{Arithmetic duality theorems}, Perspectives in Mathematics vol. \textbf{1}
J. Coates and S. Helgason, ed. Academic Press, (1986) 

\bibitem[M2]{M2} J.S. Milne,\, 
\textit{\' Etale cohomology}, Princeton Mathematical series vol. \textbf{33} 
Princeton University Press (1980)

\bibitem[Mi]{Mi} J. Milnor,\, 
\textit{Introduction to algebraic K-theory}, Annals of Math. Studies
\textbf{72}, Princeton Univ. Press (1971) 

\bibitem[MV]{MV} F. Morel, V. Voevodsky,\, 
\textit{$\A^1$-homotopy theory of schemes}, Publ. Math. Inst. Hautes {\' E}tudes Sci. 
\textbf{90}, (1999), 45-143

\bibitem[Ne]{Ne} J. Neukirch,\,  
\textit{{\" U}ber das Einbettungsproblem der algebraischen Zahlentheorie}, 
Inv. Math. \textbf{21} (1973), 59-116

\bibitem[Ng]{Ng} Q-D. Nguyen, Thong\,  
\textit{Analogues sup{\' e}rieurs du noyau sauvage}, S{\' e}m. Th{\' e}or. Nombres Bordeaux (2) 
\textbf{4} (1992), no. 2, 263-271 

\bibitem[Os]{Os} P.A. Ostvaer,\,  
\textit{Wild kernels at the prime 2}
Acta Arith. \textbf{112}, (2004) 109-130

\bibitem[Q1]{Q1} D. Quillen,\, 
\textit{Higher Algebraic K-theory: I}, Lecture Notes in Mathematics  
       \textbf{341} (1973), 85-147, Springer-Verlag

\bibitem[Q2]{Q2} D. Quillen,\, 
\textit{Finite generation of the groups $K_i$ of rings of integers}, Lecture Notes in Mathematics
        \textbf{341} (1973), 179-214, Springer-Verlag

\bibitem[Q3]{Q3} D. Quillen,\, 
\textit{On the cohomology and K-theory of the general linear groups
over a finite field}, Ann. of Math.
       \textbf{96 (2)} (1972), 552-586

\bibitem[Ri]{Ri} L. Ribes ,\, 
\textit{Introduction of profinite groups and Galois cohomology}, 
Queen's Papers in Pure and Applied Mathematics 
\textbf{24}, Queen's Univ. Kingston, Ontario (1970) 

\bibitem[RW]{RW} J. Rognes, C. Weibel,\,  
\textit{2-primary algebraic K-theory of rings of integers in number fields}
Journal of the AMS \textbf{13} no. 1, (2000) 1-54 


\bibitem[R]{R} M. Rost,\,  
\textit{Chain lemma for splitting fields of symbols}
Preprint 1998, http://www.math.uni-bielefeld.de/~rost/chain-lemma.html


\bibitem[Sch1]{Sch1} P. Schneider,\,  
\textit{Die Galoiscohomologie p-adischer Darstellungen {\" u}ber Zahlkorpern }, 
Ph.D. Thesis 1 (1980)

\bibitem[Sch2]{Sch2} P. Schneider,\,  
\textit{{\" U}ber gewisse Galoiscohomologiegruppen }, Math. Zeit.,
       \textbf{168} (1979), 181-205

\bibitem[Se]{Se} J-P. Serre ,\, 
\textit{Galois Cohomology}, Springer (1997) 
              
\bibitem[Sh]{Sh} S. Shatz,\,  
\textit{Profinite groups, arithmetic, and geometry}, Annals of 
Mathematics Studies, Princeton University Press, 1972
       \textbf{}

\bibitem[Si]{Si} J. Silverman,\,  
\textit{The arithmetic of elliptic curves}, Springer-Verlag 1986
       \textbf{}

\bibitem[So1]{So1} C. Soul\' e,\,  
\textit{K-th\' eorie des anneaux d'entiers de corps de nombres et cohomologie
\' etale}, Inv. Math. 
       \textbf{55} (1979), 251-295

\bibitem[So2]{So2} C. Soul\' e,\,  
\textit{Groupes de Chow et K-th\' eorie de vari\' et\' es sur un corps fini}, Math. Ann.
       \textbf{268} (1984), 317-345

\bibitem[Su1]{Su1} A. Suslin,\,  
\textit{On the K-theory of local fields}, J. of Pure and Appl. Algebra, 
\textbf{34} (1984), 301-318  

\bibitem[Su2]{Su2} A. Suslin,\,  
\textit{Torsion in $K_2$ of fields}, K-theory, 
\textbf{1} (1987), 5-29  

\bibitem[Ta1]{Ta1} J. Tate,\, \textit{Letter from Tate to Iwasawa on a relation 
between $K_2$  and Galois cohomology},
In Algebraic $K$-theory II, p. 524-527. 
Lecture Notes in Mathematics \textbf{342}, Berlin-Heidelberg-New York : Springer (1973)

\bibitem[Ta2]{Ta2} Tate, J.\,
\textit{Relation between $K_2$ and Galois cohomology}
Invent. Math. \textbf{36} (1976) 257-274

\bibitem[Ta3]{Ta3} Tate, J.\,
\textit{On the torsion in $K_2$ of fields},
Algebraic Number Theory, International Symposium Kyoto 
\textbf{} (1976) 243-261

\bibitem[V1]{V1} V. Voevodsky,\,  
\textit{Motivic cohomology with $\Z/2$ coefficients},
Publ. Math. Institute Hautes Etudes Sci., \textbf{98}, (2003) 59-104

\bibitem[V2]{V2} V. Voevodsky,\,  
\textit{On motivic cohomology with $\Z/l$ coefficients}
to appear in Annals of Math.  

\bibitem[VSF]{VSF} V. Voevodsky, A. Suslin, E. Friedlander,\, 
\textit{Cycles, Transfers and Motivic Homology Theories}, Annals of Math. Studies
\textbf{143}, Princeton Univ. Press (2000)

\bibitem[We1]{We1} C. Weibel,\,  
\textit{Introduction to algebraic K-theory}, book in progress at Charles 
Weibel home page, http://www.math.rutgers.edu/~weibel/
\textbf{}

\bibitem[We2]{We2} C. Weibel,\,  
\textit{Algebraic K-theory of rings of integers in local and global fields},
Handbook of K-theory, \textbf{Vol. 1}, (2005) 139-190

\bibitem[We3]{We3} C. Weibel,\,  
\textit{Higher wild kernels and divisibility in the K-theory of number fields},
J. of Pure and Appl. Algebra \textbf{206} (2006), 222-244 

\bibitem[We4]{We4} C. Weibel,\,  
\textit{The norm residue isomorphism theorem},
Journal of Topology \textbf{2} (2009), 346-372
\end{thebibliography}

\end{document}